\documentclass[letterpaper,12pt]{amsart}

\usepackage[utf8]{inputenc}
\usepackage{amsmath}
\usepackage{amssymb}
\usepackage{enumerate}
\usepackage{relsize}
\usepackage{mathtools}
\usepackage{verbatim}

\usepackage{hyperref}
\hypersetup{
    colorlinks=true,
    linkcolor=blue,
    filecolor=magenta,      
    urlcolor=blue,
}

\usepackage{tikz}
\usetikzlibrary{matrix,arrows}

\usepackage{amsthm}

\newtheorem{thm}{Theorem}[section]
\newtheorem{lem}[thm]{Lemma}
\newtheorem{prop}[thm]{Proposition}
\newtheorem{propdef}[thm]{Proposition/Definition}

\theoremstyle{remark}
\newtheorem{rmk}[thm]{Remark}

\title{Uhlenbeck compactification as a Bridgeland moduli space}
\author{Tuomas Tajakka}
\date{July 21, 2020}

\usepackage{commands}

\begin{document}

\maketitle

\begin{abstract}
    Let $(X,H)$ be a smooth, projective, polarized surface over $\C$, and let $v \in \Kn(X)$ be a class of positive rank. We prove that for certain Bridgeland stability conditions $\si = (\sA, Z)$ ``on the vertical wall'' for $v$, the good moduli space $M^\si(v)$ parameterizing S-equivalence classes of $\si$-semistable objects of class $v$ in $\sA$ is projective. Moreover, we construct a bijective morphism $M^{\mathrm{Uhl}}(v) \to M^\si(v)$ from the Uhlenbeck compactification of $\mu$-stable vector bundles.
\end{abstract}

\tableofcontents


\section{Introduction}
The aim of this paper is to prove projectivity of certain moduli spaces of Bridgeland semistable objects on a smooth, projective surface $X$ over $\C$, and relate these moduli spaces to the Uhlenbeck compactification of the moduli of $\mu$-stable vector bundles on $X$.

\subsection*{Classical background}

Constructing projective moduli spaces is an important problem in algebraic geometry. Some of the early successes in this direction were provided by Mumford who developed the machinery of Geometric Invariant Theory (GIT) for taking projective quotients of varieties by group actions, and used it to construct various moduli spaces. An important example of this method is the construction of the moduli space of slope-semistable sheaves on a curve \cite{mumford}. The \textit{slope} of a coherent sheaf $F$ on a smooth, projective curve $C$ is the rational number $\mu(F) = \deg(c_1(F))/\rk(F)$. A locally free sheaf $F$ is \textit{semistable} if for every proper nonzero subsheaf $E \subs F$ the inequality $\mu(E) \le \mu(F)$ holds, \textit{stable} if $\mu(E) < \mu(F)$, and \textit{polystable} if $F \cong \oplus_i F_i$ where the $F_i$ are stable bundles of the same slope. Every semistable sheaf $F$ has a \textit{Jordan-H\"older filtration} by stable sheaves $F_i$, and the polystable sheaf $\gr(F) = \oplus_i F_i$ is called the \textit{associated graded} of $F$. Two semistable sheaves $F$ and $F'$ are \textit{S-equivalent} if $\gr(F) \cong \gr(F')$, and every S-equivalence class contains a unique polystable sheaf. The projective moduli space Mumford constructed parameterizes S-equivalence classes of semistable sheaves, or equivalently isomorphism classes of polystable sheaves, and contains the locus of stable sheaves as an open subscheme.

The notion of slope-stability has been generalized to a higher dimensional smooth, projective, polarized variety $(X, H)$ in various ways. One successful notion is Geiseker-stability, where the slope is replaced by the reduced Hilbert polynomial which takes into account all Chern classes. Moduli spaces parameterizing S-equivalence classes of Gieseker-semistable sheaves were constructed using GIT by Gieseker, Maruyama, and Simpson. However, unlike in the case of a curve, to obtain a projective moduli space, some non-locally free coherent sheaves must be included in the moduli problem.

A direct generalization of slope-stability to a higher dimensional $X$ is obtained by modifying the formula to be $\mu(F) = (c_1(F) \cdot H)/\rk(F)$. Although Jordan-H\"older filtrations exist, and $\mu$-stability has many other useful properties, this notion of stability does not allow for a moduli space parameterizing S-equivalence classes. However, when $X$ is a surface, a related projective scheme parameterizing sheaves up to a coarser equivalence relation was constructed by Li \cite{li} following work of Uhlenbeck and Donaldson in gauge theory. This so-called Uhlenbeck compactification contains the moduli of $\mu$-stable vector bundles as an open subscheme. Two $\mu$-semistable sheaves $F_1$ and $F_2$ are identified in the Uhlenbeck compactification precisely when $\gr(F_1)^{\vee\vee} \cong \gr(F_2)^{\vee\vee}$ and the 0-dimensional sheaves $\gr(F_1)^{\vee\vee}/\gr(F_1)$ and $\gr(F_2)^{\vee\vee}/\gr(F_2)$ are supported at the same points of $X$ with the same lengths. 

\subsection*{Bridgeland stability}

A more general notion of stability was introduced by Bridgeland \cite{bridgeland} following the work of Douglas in string theory. Bridgeland's idea is to extend stability from coherent sheaves to objects in the derived category $D^b(X)$. A \textit{Bridgeland stability condition} on $X$ is a pair $\si = (\sA, Z)$ consisting of a heart of a bounded t-structure $\sA \subs D^b(X)$ and a group homomorphism 
\[ Z: K(\sA) \to \C \]
that gives rise to a slope function on $\sA$. The set of all such stability conditions naturally forms a complex manifold endowed with interesting wall-and-chamber structures. It was soon realized that Bridgeland stability is suitable for studying the birational geometry of classical moduli spaces. Namely, the moduli space of Bridgeland semistable objects remains constant within each open chamber, and moduli spaces corresponding to open chambers separated by a wall are frequently birational. A prominent example of this approach is the complete description of the minimal model program of the Hilbert scheme of points on a surface \cite{ABCH}.

Although constructing stability conditions on higher dimensional varieties remains an important open problem, a general method for producing stability conditions on a surface $X$ was developed by Bridgeland \cite{bridgelandK3} for K3 surfaces, and extended by Arcara and Bertram \cite{ABL13} for all surfaces. Moreover, a moduli space $M^\si(v)$ parameterizing S-equivalence classes of $\si$-semistable objects in $\sA$ of numerical class $v$ exists as a proper algebraic space \cite[Theorem 7.25, Example 7.27]{AHLH} and possesses a natural nef divisor class $l_\si$ that varies with the stability condition $\si$ \cite{BM}. Despite this, projectivity of Bridgeland moduli spaces is known only in few cases.

\subsection*{Statement of results}

The goal of this work is to prove projectivity of the Bridgeland moduli space when $X$ is an arbitrary smooth, projective surface, and $\si$ lies on the ``vertical wall'' bounding the chamber corresponding to Gieseker stability. The main results are Theorem \ref{projectivity} and Theorem \ref{uhlenbeck} and can be summarized as follows.
\begin{thm}
    Let $(X,H)$ be a smooth, projective, polarized surface over $\C$, and let $v \in \Kn(X)$ be a numerical class of positive rank. There exists a Bridgeland stability condition $\si = (\sA, Z)$ with the following properties.
    \begin{enumerate}[(a)]
        \item The $\si$-polystable objects of class $v$ in $\sA$ are of the form
        \[ E = F \oplus \left(\bigoplus_i \Oh_{p_i}^{\oplus n_i} [-1]\right), \]
        where $F$ is a $\mu$-polystable locally free sheaf and the $\Oh_{p_i}$ are structure sheaves of closed points $p_i \in X$.
        \item The good moduli space $M^\si(v)$ parameterizing $\si$-polystable objects is projective and the class $l_\si$ is ample.
        \item There is a bijective morphism $M^{\mathrm{Uhl}}(v) \to M^\si(v)$ from the Uhlenbeck compactification of $\mu$-stable locally free sheaves.
    \end{enumerate}
\end{thm}
Our main mathematical contribution is part (b) of the theorem. The idea is to show that $l_\si$ is ample by directly producing sections. The technique is a combination of Li's construction of the Uhlenbeck compactification and an argument of Seshadri to produce sections of determinantal line bundles on the moduli space of slope-semistable vector bundles on a curve \cite{seshadri}.

Part (c) of the theorem follows straightforwardly by comparing the proof of part (b) and Li's construction of the Uhlenbeck compactification. To convince the reader that the bijection is plausible, consider a $\mu$-polystable torsion-free sheaf $F$ on $X$. The sheaf $F$ fits into a short exact sequence
\[ 0 \to F \to F^{\vee\vee} \to T \to 0, \]
where $F^{\vee\vee}$ is a $\mu$-polystable locally free sheaf and $T$ has 0-dimensional support. Recall that the Uhlenbeck compactification records the information of $F^{\vee\vee}$ together with the length $l_p(T)$ of $T$ at closed points $p \in X$. On the other hand, the above exact sequence rotates to the exact sequence
\[ 0 \to T[-1] \to F \to F^{\vee\vee} \to 0 \]
in the heart $\sA \subs D^b(X)$. The containment $T[-1] \subs F$ is part of the Jordan-H\"older filtration with respect to $\si$, and in fact $F$ is S-equivalent to the $\si$-polystable object
\[ F^{\vee\vee} \oplus \left(\bigoplus_{p \in X} \Oh_p^{\oplus l_p(T)}[-1]\right). \]

Part (a) of the theorem is known to experts in general and worked out by Lo and Qin in \cite{LQ} in the case when $\rk(v)$ and $H \cdot c_1(v)$ are coprime, where the authors also observe the set theoretic bijection with the Uhlenbeck compactification. Although we include a proof in the general case, we claim no originality.

\subsection*{Relation to previous work}
Bridgeland moduli spaces on surfaces are known to be projective in only some cases:
\begin{itemize}
    \item When $X$ is either $\p^2$ \cite{ABCH} or $\p^1 \times \p^1$ or the blow-up of $\p^2$ at a point \cite{AM}, all Bridgeland moduli spaces can be related to moduli of quiver representations. Similar results are conjectured to hold for all Del Pezzo surfaces \cite{AM}.
    
    \item For an arbitrary surface $X$, stability in a special chamber coincides with Giekeser stability. See \cite{bridgelandK3} for the case of a K3 surface.
    
    \item When $X$ is a K3 or abelian surface of Picard rank 1, and $v$ is the class of certain 1-dimensional sheaves on $X$, and $\si$ is \textit{generic} with respect to $v$, meaning that it does not lie on a wall for $v$, Arcara and Bertram \cite{ABL13} construct moduli spaces $M^\si(v)$ as iterated Mukai flops of the Gieseker moduli space of class $v$.
    
    \item When $X$ is an abelian surface of Picard rank 1, Maciocia and Meachan \cite{MM} construct moduli spaces $M^\si(v)$ for certain classes $v$ of rank 1 and when $\si$ is generic for $v$ by relating $M^\si(v)$ to Gieseker moduli spaces via a Fourier-Mukai transform.
    
    \item When $X$ is a K3 surface and $\si$ is generic with respect to $v$, Bayer and Macr\`i show in \cite{BM}, generalizing similar results for K3 and abelian surfaces in \cite{mmy}, that $M^\si(v)$ is projective by relating $\si$-stability on $X$ to Gieseker stability on a Fourier-Mukai partner.
    
    \item When $X$ is an unnodal Enriques surface and $\si$ is generic with respect to $v$, Nuer shows in \cite{nuer} that the moduli space $M^\si(v)$ is projective by producing a finite map to a related Bridgeland moduli space on the K3 universal cover of $X$.
    
\end{itemize} 
While in all of these cases the projectivity of the moduli spaces ultimately rely on GIT constructions, and the line bundle $l_\si$ of Bayer and Macr\`i can be seen to be ample after the fact, a general GIT framework for Bridgeland stability is currently unavailable. Our method avoids the use of GIT and proves ampleness of $l_\si$ by directly producing enough sections. To our knowledge, this is the first example of a Bridgeland moduli space on a surface whose projectivity does not rely on GIT.

The relationship between $M^{\mathrm{Uhl}}(v)$ and $M^\si(v)$ when $\si$ lies on the vertical wall for $v$ has been previously studied by Lo in \cite{lo}, whose result together with properness of the good moduli space implies that when $X$ is a K3 surface, the good moduli space of $\si$-semistable objects is projective. Lo achieves this by relating $M^\si(v)$ to a moduli space of $\mu$-stable locally free sheaves on a Fourier-Mukai partner of $X$. Our results subsumes Lo's results and avoids the use of a Fourier-Mukai transform.

\subsection*{Open questions}
The following are potential next questions in the direction of this paper.
\begin{itemize}
    \item What is the local geometry of $M^\si(v)$, and is the morphism $M^{\mathrm{Uhl}}(v) \to M^\si(v)$ an isomorphism?
    \item What kind of birational surgery does the Gieseker moduli space undergo when we cross the vertical wall? Based on earlier work in the subject, stability on the other side of the wall should correspond to Gieseker stability under the derived dual functor.
    \item Can the methods of this paper be adapted to showing projectivity of more general Bridgeland moduli spaces on surfaces, or some other moduli spaces of sheaves or complexes on varieties? It would even be interesting to construct Gieseker moduli spaces without GIT. 
\end{itemize}

\subsection*{Acknowledgements}

The author would like to thank his advisor Jarod Alper for suggesting this problem and for constant guidance throughout the project, as well as Benjamin Schmidt, Aaron Bertram, and Max Lieblich for useful discussions. The original idea for the problem came via indirect communication from Emanuele Macr\`i.


\section{Bridgeland stability}
In this section we recall definitions and basic notions concerning Bridgeland stability. An excellent exposition of the material is \cite{MS}. 

Let $X$ be a smooth, projective variety, and let $\Kn(X)$ denote its numerical Grothendieck group, that is, the quotient of $K(X)$ by the kernel of the Euler pairing $\chi(-,-)$. A \textbf{(numerical) stability condition} on $X$ is a pair $\si = (\sA, Z)$, where \begin{itemize}
    \item $\sA \subs D^b(X)$ is the heart of a bounded t-structure, and
    \item $Z: \Kn(X) \to \C$ is a {\bf stability function} on $\sA$, that is, a group homomorphism such that for every nonzero object $A \in \sA$, we have 
    \[ Z(A) \in \Hb = \Hh \cup \R_{<0} = \{ r e^{i\pi\phi} \in \C \;|\; r > 0, \; 0 < \phi \le 1 \}. \]
\end{itemize} 
This lets us define a notion of stability in the abelian category $\sA$: we say $A \in \sA$ is \textbf{stable} (resp. \textbf{semistable}) if for every proper nonzero subobject $A' \subs A$, we have
\[ \nu_Z(A') < \nu_Z(A) \qquad (\mathrm{resp.} \quad \nu_Z(A') \leq \nu_Z(A)), \]
where 
\[ \nu_Z(A) = \begin{cases} -\frac{\re Z(A)}{\im Z(A)} & \mathrm{if} \quad \im Z(A) > 0 \\ +\infty & \mathrm{if} \quad \im Z(A) = 0. \end{cases} \]
With this notion of stability, the pair $\si = (\sA, Z)$ must satisfy the following conditions:
\begin{enumerate}[(i)]
    \item Every nonzero $A \in \sA$ has a {\bf Harder-Narasimhan filtration}
    \[ 0 = A_0 \subsetneq A_1 \subsetneq \cdots \subsetneq A_{m-1} \subsetneq A_m = A, \]
    where each quotient $F_i = A_i/A_{i-1}$ is semistable and
    \[ \nu_Z(F_1) > \cdots > \nu_Z(F_m).  \]
    \item {\bf Support property}: there is a symmetric bilinear form $Q$ on $\Kn(X) \otimes \R$ that is negative definite on the kernel of $Z$, and $Q(A, A) \ge 0$ for every semistable object $A \in \sA$.
\end{enumerate}

The set $\Stab(X)$ of stability conditions on $X$ has a natural topology with respect to which the map 
\[ \Stab(X) \to \Hom(\Kn(X), \C), \quad (\sA, Z) \mapsto Z \]
is a local homeomorphism. Moreover, for a given numerical class $v \in \Kn(X)$ there is a locally finite collection of real codimension 1 walls inside $\Stab(X)$ such that the sets of stable and semistable objects remains constant when $\si$ varies within a connected component of the complement of the walls. 

If $C$ is a curve, an example of a stability condition on $C$ is given by $\si = (\Coh(C), -\deg + i \rk)$, giving rise to the classical Mumford slope. However, if $\dim X \ge 2$, the standard heart $\Coh(X) \subs D^b(X)$ can never be the heart of a stability condition. 


\subsection{Stability conditions on surfaces}\label{section:stabcondsurf}
We now recall a construction of stability conditions on the derived category of a smooth, projective surface $X$ equipped with a very ample divisor $H$. This is achieved by tilting the standard heart $\Coh(X) \subs D^b(X)$ with respect to $\mu$-stability. 

Fix a real divisor class $B \in N^1(X)_\R$. The {\bf $B$-twisted Chern character} is defined by
\[ \ch^B = e^{-B}\cdot \ch, \]
with graded pieces
\[ \ch^B_0 = \ch_0 = \rk, \quad \ch^B_1 = \ch_1 - B \cdot \ch_0, \quad \ch^B_2 = \ch_2 - B \cdot \ch_1 + \frac{B^2}{2} \ch_0. \]
Define the {\bf $B$-twisted slope function} on $\Coh(X)$ by
\[ \mu_B(E) = \frac{H \cdot \ch^B_1(E)}{H^2 \cdot \ch^B_0(E)} = \frac{H \cdot \ch_1(E)}{H^2 \cdot \ch_0(E)} - \frac{H \cdot B}{H^2} \]
if $\rk(E) > 0$, and $\mu_B(E) = \infty$ if $\rk(E) = 0$, i.e. $E$ is a torsion sheaf. Note that this differs from the usual slope function $\mu$ only by the additive constant $-H\cdot B/H^2$ and hence defines the same notion of stability on $\Coh(X)$. 

Since Harder-Narasimhan filtrations into $\mu$-semistable factors exist, for every real number $\be$ we obtain a torsion pair on $\Coh(X)$ by setting
\begin{align*}
    \sT_\be & = \{ E \in \Coh(X) \;|\; \mu_B(F) > \be \mathrm{\;for\;every\;semistable\;factor\;} F \mathrm{\;of\;} E \}, \\
    \sF_\be & = \{ E \in \Coh(X) \;|\; \mu_B(F) \le \be \mathrm{\;for\;every\;semistable\;factor\;} F \mathrm{\;of\;} E \}.
\end{align*}
Thus, we obtain a heart $\Coh^\be(X) = \langle \sF_\be[1], \sT_\be \rangle \subs D^b(X)$ as the full subcategory whose objects are precisely those $E \in D^b(X)$ fitting in an exact triangle
\[ F[1] \to E \to T, \]
where $F \in \sF_\be, T \in \sT_\be$. 

For any $\al \in \R_{>0}$ we define a map $Z_{\al,\be}: \Kn(X) \to \C$ by setting
\begin{align*}
    Z_{\al, \be}(E) & = -\int_X e^{-(\be + i \al)H} \ch^B(E) \\
    & = \frac{\al^2 - \be^2}{2} H^2 \ch^B_0(E) + \be H\cdot\ch^B_1(E) - \ch^B_2(E) \\
    & \quad  + i\al(H \cdot \ch^B_1(E) - \be H^2 \ch^B_0(E)).
\end{align*}
We denote the associated slope function on $\Coh^\be(X)$ by $\nu_{\al,\be}$. It is shown in \cite{bridgelandK3} and \cite{ABL13} that the pair $\si_{\al,\be} = (\Coh^\be(X), Z_{\al,\be})$ is a stability condition on $X$. 

An object $E \in \Coh^\be(X)$ is called \textbf{polystable} with respect to $\si_{\al,\be}$ if 
\[ E \cong \bigoplus_i E_i \] 
where $E_i \in \Coh^\be(X)$ is $\si_{\al,\be}$-stable and $\nu_{\al,\be}(E_i) = \nu_{\al,\be}(E)$ for each $i$. Every $\si_{\al,\be}$-semistable object $E \in \Coh^\be(X)$ has a \textbf{Jordan-H\"older filtration}
\[ 0 = E_0 \subsetneq E_1 \subsetneq \ldots \subsetneq E_{m-1} \subsetneq E_m = E \]
where the successive quotients $E_i/E_{i-1}$ are $\si_{\al,\be}$-stable with 
\[ \nu_{\al,\be}(E_i/E_{i-1}) = \nu_{\al,\be}(E) \] 
for $i = 1,\ldots,m$. The \textbf{associated graded object} of $E$ is the direct sum 
\[ \gr(E) = \bigoplus_i E_i/E_{i-1}, \] 
unique up to noncanonical isomorphism, and two semistable objects $E$ and $E'$ are \textbf{S-equivalent} if $\gr(E) \cong \gr(E')$.

We will need the following observation. 
\begin{lem}\label{ss-ses}
    If $E \in \Coh^\be(X)$ and $\im Z_{\al,\be}(E) = 0$, then in the above triangle 
    \[ F[1] \to E \to T, \]
    $T$ has 0-dimensional support, and $F$ is a $\mu$-semistable sheaf with $\mu_B(F) \\ = \be$.
\end{lem}
\begin{proof}
    Note that for a coherent sheaf $G$, we have $\im Z_{\al,\be}(G) > 0$ (resp. $= 0$) if and only if $\mu_B(G) > \be$ (resp. $= \be$). It follows from the construction that for any $E \in \Coh^\be(X)$,
    \[ \im Z_{\al,\be}(E) = \im Z_{\al,\be}(T) - \im Z_{\al,\be}(F) \]
    and $\im Z_{\al,\be}(T), -\im Z_{\al,\be}(F) \ge 0$. 
    
    First, if $T$ has positive rank, then by assumption $\mu_B(T) > \be$, and so $\im Z_{\al,\be}(E) > 0$. Hence $T$ must have rank 0. If the support of $T$ is 1-dimensional, then \[ \im Z_{\al,\be}(T) = \al H \cdot \ch^B_1(E) = \al H \cdot \ch_1(E) > 0 \]
    since $H$ is ample. This means that $T$ must have 0-dimensional support.
    
    Second, if $F$ is not $\mu$-semistable, then it has a Harder-Narasimhan filtration
    \[ 0 \neq F_1 \subset F_2 \subset \cdots \subset F_{m-1} \subset F_m = F \]
    with respect to $\mu$ with $m \ge 2$, and 
    \[ \im Z_{\al,\be}(F_i/F_{i-1}) \le 0 \quad \mathrm{for\;all}\; i, \quad \mathrm{and} \quad \im Z_{\al,\be}(F_m/F_{m-1}) < 0. \] 
    But then
    \[ - \im Z_{\al,\be}(F) = - \sum_{i=1}^m \im Z_{\al,\be}(F_i/F_{i-1}) > 0. \]
    Thus, $F$ must be $\mu$-semistable.
\end{proof}

\subsection{Wall-and-chamber structure}
We visualize the stability conditions $\si_{\al,\be} = (\Coh^\be(X), Z_{\al,\be})$ as living in the upper half-plane with horizontal $\be$-axis and vertical $\al$-axis. The wall-and-chamber structure in the $(\al,\be)$-plane was analyzed in \cite{maciocia} and turns out to be rather simple. If $v \in \Kn(X)$ is a class of positive rank, then there is a unique vertical wall at 
\[ \be_0 = \frac{H \cdot \ch_1^B(v)}{H^2 \ch_0(v)}, \] and on each side of the vertical wall, there is a nested sequence of semicircles with center on the $\be$-axis and contained in a largest semicircle. In particular, on either side of the vertical wall, there is an unbounded open chamber. The walls in the $(\al,\be)$-plane are pairwise disjoint, and so there is no further wall-and-chamber decomposition within each wall. In the unbounded open chamber left of the vertical wall, $\si$-stability coincides with Gieseker stability.
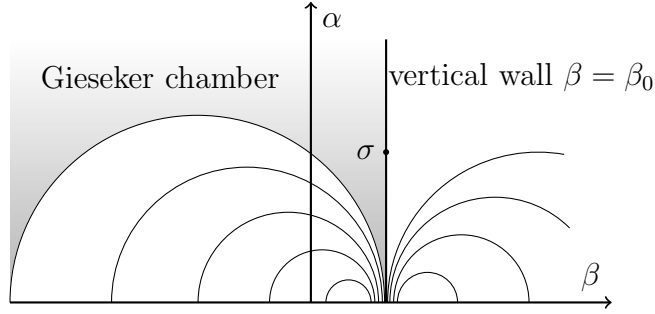
\begin{center}
\begin{tikzpicture}

\shadedraw[bottom color=black!30!white,top color=white,draw=white] (-4,0) arc (180:0:2.49) -- (0,0) -- (1,0) -- (1,3.5) -- (-4,3.5);

\draw[thick,->] (-4,0) -- (4,0) node[anchor=south east] {$\be$};
\draw[thick,->] (0,0) -- (0,4) node[anchor=north west] {$\al$};

\draw[thick] (1,0) -- (1,3.5);

\draw[] (0.98,0) arc (0:180:2.49);
\draw[] (0.95,0) arc (0:180:1.8);
\draw[] (0.9,0) arc (0:180:1.2);
\draw[] (0.85,0) arc (0:180:0.7);
\draw[] (0.8,0) arc (0:180:0.3);

\draw[] (1.02,0) arc (180:80:2);
\draw[] (1.05,0) arc (180:45:1.4);
\draw[] (1.1,0) arc (180:0:0.9);
\draw[] (1.15,0) arc (180:0:0.4);

\draw (-2,3) node {Gieseker chamber};
\draw (2.8,3) node {vertical wall $\be = \be_0$};
\fill (1,2) circle (0.04) node[anchor=east] {$\si$};

\end{tikzpicture}
\end{center}
Our goal is to study the moduli space of semistable objects when $\si$ lies on the vertical wall.

\subsection{Stability on the vertical wall}\label{section:stabvertwall}
In this subsection we classify stable and semistable objects on the vertical wall. Let $v \in \Kn(X)$ be a class of positive rank, and let 
\[ (\Coh^{\be_0}(X), Z_{\al,\be_0}) \in \Stab(X) \] 
be the stability condition constructed in the previous section with 
\[ \be_0 = \frac{H \cdot \ch_1^B(v)}{H^2 \ch_0(v)} \quad \mathrm{and} \quad \al > 0. \] Note that since other walls do not intersect the vertical wall, the sets of stable and semistable objects are independent of $\al$.

We first note that there are no nonzero objects in $\Coh^{\be_0}(X)$ with numerical class $v$. Namely, by Lemma \ref{ss-ses}, any object $E \in \Coh^{\be_0}(X)$ with $\im Z_{\al,\be_0}(E) = 0$ fits in a triangle
\[ F[1] \to E \to T, \]
where $T$ has 0-dimensional support and $F$ is $\mu$-semistable. But $\rk(E) = \rk(T) + \rk(F[1]) = -\rk(F) \le 0$, while $\rk(v) > 0$ by assumption. Therefore, it is convenient to instead consider the stability condition 
\[ \si = (\sA, Z), \quad \mathrm{where} \quad \sA = \Coh^{\be_0}(X)[-1], \quad Z = -Z_{\al,\be_0}. \] 
Note that this does not change the slope function $\nu = -\re Z/\im Z$. By definition, the heart $\sA$ consists of objects $E$ with fitting in a triangle
\[ F \to E \to T[-1] \]
where $F \in \sF_{\be_0}, T \in \sT_{\be_0}$.

The next proposition gives a description of stable and semistable objects in $\sA$ of slope $\nu = \infty$ with respect to $\si$. This includes objects of class $v \in \Kn(X)$. Part (i) will be crucial in the proof that the moduli space $M^\si(v)$ of semistable objects of class $v$ is projective, and part (iii) will let us identify $M^\si(v)$ with the Uhlenbeck compactification of the moduli of $\mu$-stable vector bundles, at least on the level of points.
\begin{prop}\label{ss-object-vertical-classification}
    Let 
    \[ \si = (\sA, Z), \quad \mathrm{where} \quad \sA = \Coh^{\be_0}(X)[-1], \quad Z = -Z_{\al,\be_0}. \]
    \begin{enumerate}[(i)]
        \item Any object $E \in \sA$ with $\nu(E) = \infty$ is $\si$-{\bf semistable} and fits in a triangle
        \[ F \to E \to T[-1] \]
        where $T$ is a sheaf supported in dimension 0, and $F$ is a $\mu$-semistable sheaf of slope $\mu_B(F) = \be_0$. 
        \item An object $E \in \sA$ with $\nu(E) = \infty$ is $\si$-{\bf stable} if and only if in the above triangle either $F$ is a $\mu$-stable locally free sheaf and $T = 0$, or $T = \Oh_p$ is the structure sheaf of a closed point $p \in X$ and $F = 0$.
        \item An object $E \in \sA$ of class $v$ is $\si$-{\bf polystable} if and only if
        \[ E \cong \left(\bigoplus_i F_i\right) \oplus \left(\bigoplus_j \Oh_{p_j}[-1]\right), \]
        where each $F_i$ is a $\mu$-stable locally free sheaf of slope $\mu = \mu_B(v)$, and each $p_j \in X$ is a closed point.
    \end{enumerate}
\end{prop}
Part (i) of Proposition \ref{ss-object-vertical-classification} follows from Lemma \ref{ss-ses} and the constructions, while part (iii) follows from part (ii) by the definition of polystability. We prove part (ii)
in a series of lemmas below. Lemmas \ref{muStableLfIsSigmaStable} and \ref{skyscraperIsSigmaStable} show that $\mu$-stable locally free sheaves and shifted skyscraper sheaves are $\si$-stable, and Lemmas \ref{sigmaStablePosRkIsMuStable} and \ref{sigmaStableRk0isSkyscraper} show the converse.
\begin{lem}\label{muStableLfIsSigmaStable}
    A $\mu$-stable locally free sheaf $E$ with $\rk(E) > 0$ and $\mu_B(E) = \be_0$ is $\si$-stable.
\end{lem}
\begin{proof}
    Let $F \hookrightarrow E$ be an inclusion in $\sA$ with $\nu(F) = \infty$. We must show $F = 0$ or $F = E$. The induced short exact sequence
    \[ 0 \to F \to E \to G \to 0 \]
    is by definition an exact triangle in $D^b(X)$ with each vertex in $\sA$. Cohomology with respect to the standard t-structure leads to an exact sequence
    \[ 0 \to \sH^0(F) \to E \to \sH^0(G) \to \sH^1(F) \to 0 \to \sH^1(G) \to 0 \]
    of sheaves. This immediately implies that $\sH^1(G) = 0$, i.e. $G = \sH^0(G)$, and that $\sH^0(F)$ is a subsheaf of $E$.
    
    We have three cases.
    \begin{itemize}
        \item If $\sH^0(F) \xrightarrow{\sim} E$, then $G \xrightarrow{\sim} \sH^1(F)$. But $G \in \sF_{\be_0}, \sH^1(F) \in \sT_{\be_0}$, so we must have $G = \sH^1(F) = 0$, hence $F = \sH^0(F) = E$.
        
        \item If $\sH^0(F)$ is a proper, nonzero subsheaf of $E$, then by the assumption on $E$ we have $\mu_B(\sH^0(F)) < \mu_B(E) = \be_0$. Let $N$ denote the image of the map $E \to G$, so that we have the short exact sequences
        \[ 0 \to \sH^0(F) \to E \to N \to 0 \]
        and 
        \[ 0 \to N \to G \to \sH^1(F) \to 0. \]
        By assumption, $\nu(G) = \infty$ so $G$ is $\mu$-semistable, and thus $\mu_B(N) \le \mu_B(G)$. This gives the absurd inequality 
        \[ \be_0 = \mu_B(E) < \mu_B(N) \le \mu_B(G) \le \mu_B(\sH^1(F)) \le \be_0. \]
        Thus, this case is impossible.
        
        \item If $\sH^0(F) = 0$, then $F = \sH^1(F)[-1]$. Denote $F' = \sH^1(F)$. The short exact sequence
        \[ 0 \to E \to G \to F' \to 0 \]
        implies 
        \[ H\cdot \ch^B_1(G) = H\cdot \ch^B_1(E) + H\cdot \ch^B_1(F'). \]
        Assume for contradiction that $\rk(F') > 0$. Since $\rk(G) \ge \rk(E) > 0$ by assumption, from the above equality and the definition of $\mu_B$ we obtain
        \begin{align*} \be_0 \rk(E) + \mu_B(F') \rk(F') & = \mu_B(E) \rk(E) + \mu_B(F') \rk(F') \\
        & = \mu_B(G) \rk(G) \\
        & \le \be_0 \rk(G), \end{align*}
        so that
        \[ \mu_B(F') \rk(F') \le \be_0(\rk(G) -  \rk(E)) = \be_0 \rk(F'). \]
        However, since $\mu_B(F') > \be_0$, this inequality is impossible. Thus, $\rk(F') = 0$, which also implies $\rk(E) = \rk(G)$ and $\ch^B_1(F') = \ch_1(F')$. 
        
        Next, assume for contradiction that $F'$ has 1-dimensional support. Since $H$ is ample, this implies $H \cdot \ch^B_1(F') > 0$. But on the other hand, 
        \[ \mu_B(G) H^2 \rk(G) = \mu_B(E) H^2 \rk(E) + H \cdot \ch^B_1(F'), \]
        so that
        \[ H \cdot \ch^B_1(F') = H^2 (\mu_B(G) - \be_0)\rk(E) \le 0 \]
        since by assumption $\mu_B(G) \le \be_0$, again a contradiction. Thus, $F'$ has 0-dimensional support. Now if $F' \neq 0$, then we have a locally free subsheaf $E$ of a torsion-free sheaf $G$ with 0-di\-men\-sion\-al quotient $F'$. But as mentioned in \cite[Example 1.1.16]{HL}, the quotient $G/E$ has no 0-dimensional associated points. Thus, we must have $F' = 0$.
    \end{itemize} 
\end{proof}

\begin{lem}\label{skyscraperIsSigmaStable}
    The shifted skyscraper sheaf $\Oh_p[-1]$ is $\si$-stable for every closed point $p \in X$.
\end{lem}
\begin{proof}
    Let $F \hookrightarrow \Oh_p[-1]$ be an inclusion in $\sA$ with $\nu(F) = \infty$. We must show $F = 0$ or $F = \Oh_p[-1]$. Like above, the induced short exact sequence
    \[ 0 \to F \to \Oh_p[-1] \to G \to 0 \]
    in $\sA$ yields the exact sequence
    \[ 0 \to \sH^0(G) \to F \to \Oh_p \to \sH^1(G) \to 0 \]
    of sheaves, and $F = \sH^1(F)$. If $\Oh_p \xrightarrow{\sim} \sH^1(G)$, then $\sH^0(G) \cong F$, and since $\sH^0(G) \in \sF_\be, F \in \sT_\be$, we have $F = 0$.
    
    If on the other hand $\sH^1(G) = 0$, then $G = \sH^0(G)$, and the short exact sequence
    \[ 0 \to G \to F \to \Oh_p \to 0 \]
    implies that $\mu_B(G) = \mu_B(F)$, and we once again see that $G = F = 0$, a contradiction.
\end{proof}


\begin{lem}\label{sigmaStablePosRkIsMuStable}
    If $E \in \sA$ is a $\si$-stable object with $\rk(E) > 0$ and $\nu(E) = \infty$, then $E$ is a $\mu$-stable locally free sheaf.
\end{lem}
\begin{proof}
    The object $E$ fits in an exact triangle
    \[ F \to E \to T[-1] \]
    where $F$ is a $\mu$-semistable torsion-free sheaf with $\mu_B(F) = \be_0$ and $T$ is a 0-dimensional sheaf. If $T \neq 0$, then $F$ is a destabilizing subobject of $E$ in $\sA$ unless $F = 0$, in which case $\rk(E) = -\rk(T) = 0$, contrary to the assumption. Thus, $T = 0$ and $E$ is a $\mu$-semistable sheaf.
    
    We next show that $E$ is locally free. Since $E$ is torsion-free, the canonical evaluation map $E \to E^{\vee \vee}$ is injective with cokernel $Q$ supported in dimension 0. Now $E^{\vee\vee}$ and $Q[-1]$ both lie in the heart $\sA$, so the short exact sequence 
    \[ 0 \to E \to E^{\vee \vee} \to Q \to 0 \]
    of coherent sheaves gives an exact sequence
    \[ 0 \to Q[-1] \to E \to E^{\vee \vee} \to 0 \]
    in $\sA$. Since $\nu(Q[-1]) = \nu(E) = \infty$ and $E$ is stable, we must have $Q = 0$, and so $E \cong E^{\vee\vee}$ is locally free.
    
    To show that $E$ is $\mu$-stable, let
    \[ 0 \subset E_1 \subset \cdots \subset E_r = E \]
    be a Jordan-H\"older filtration into $\mu$-stable factors. If $r > 1$, then $E/E_1$ is a $\mu$-semistable sheaf with $\mu_B(E/E_1) = \be_0$, so the short exact sequence of sheaves
    \[ 0 \to E_1 \to E \to E/E_1 \to 0 \]
    is also a short exact sequence in $\sA$ of objects with $\nu = \infty$, which contradicts the $\si$-stability of $E$. Thus, $E = E_1$ is $\mu$-stable.
\end{proof}

\begin{lem}\label{sigmaStableRk0isSkyscraper}
    If $E \in \sA$ is $\si$-stable with $\rk(E) = 0$ and $\nu(E) = \infty$, then $E = \Oh_p[-1]$ for some closed point $p \in X$.
\end{lem}
\begin{proof}
    From the triangle
    \[ F \to E \to T[-1] \]
    as above, we get 
    \[ 0 = \rk(E) = \rk(F) - \rk(T) = \rk(F), \]
    so since $F$ is torsion-free, we must have $F = 0$, and $E = T[-1]$ is the shift of a 0-dimensional sheaf. But any proper subsheaf $T' \subs T$ is also 0-dimensional, so $T'[-1] \in \sA$ is a destabilizing subobject of $E$ with respect to $\si$. Thus, $T$ must have length 1, and so $T = \Oh_p$ for some $p \in X$.
\end{proof}

\begin{rmk}
Proposition \ref{ss-object-vertical-classification} can also be deduced from \cite[Proposition 2.2]{huy} as follows. Any stable object $E \in \sA$ is minimal, since a nonzero surjection $E \twoheadrightarrow E'$ in $\sA$ implies $\nu(E') > \nu(E)$ unless $E' = E$. Conversely any minimal object is automatically stable. Although \cite[Proposition 2.2]{huy} is stated in the case when $X$ is a K3 surface, the proof works for any surface.
\end{rmk}


\section{Moduli of semistable objects}
In this section we overview some definitions and results regarding moduli spaces of Bridgeland semistable objects. 

\subsection{Moduli stacks}
Let $X$ be a smooth, projective surface over $\C$ with a very ample divisor $H$, let $v \in \Kn(X)$ be a numerical class, and let $\si = (\sA, Z) \in \Stab(X)$ be a Bridgeland stability condition on $X$.

Define a category fibered in groupoids $\sM^\si(v)$ over the big \'etale site of $\C$-schemes as follows. The objects of $\sM^\si(v)$ are pairs $(S, \sE)$, where $S$ is a scheme over $\C$, and $E \in D^b(S \times X)$ is a complex of coherent sheaves relatively perfect over $S$, and whenever $S$ is of finite type over $\C$, for every closed point $s \in S$, the derived restriction of $E$ to the fiber $\{s\} \times X \cong X$ lies in $\sA$, is $\si$-semistable, and has numerical class $v$. A morphism $(S', \sE') \to (S, \sE)$ in $\sM^\si(v)$ is a pair $(f, f^\sharp)$, where $f: S' \to S$ is a morphism of $\C$-schemes, and $f^\sharp: f_* \sE' \to \sE$ is a morphism in $D^b(S \times X)$ whose adjoint is an isomorphism $\sE' \xrightarrow{\sim} f^* \sE$ in $D^b(S' \times X)$.

If $\si$ is obtained by tilting with respect to $\mu$-stability as in Section \ref{section:stabcondsurf}, then based on work in \cite{lie06}, \cite{ABL13}, and \cite{AP06}, it is proved in \cite{toda08} that $\sM^\si(v)$ is an algebraic stack of finite type over $\C$. 

\subsection{Good moduli spaces}
A good moduli space is a generalization to algebraic stacks of the usual coarse moduli space associated to a Deligne-Mumford stack or a gerbe. In a sense, a good moduli space is a scheme or algebraic space that most closely approximates an algebraic stack. Based on ideas from Geometric Invariant Theory, Alper gave the definition and developed basic properties of good moduli spaces in \cite{AlperGMS}.

Let $\sM$ be an algebraic stack. A quasi-compact, quasi-separated morphism $\pi: \sM \to M$ to an algebraic space $M$ is called a {\bf good moduli space}, if 
\begin{enumerate}[(i)]
    \item the pushforward functor $\pi_*: \Qcoh(\sM) \to \Qcoh(M)$ is exact, and
    \item the natural map $\Oh_M \to \pi_*\Oh_\sM$ is an isomorphism.
\end{enumerate}
We list a few basic properties of good moduli spaces.
\begin{prop}
    If $\pi: \sM \to M$ is a good moduli space, then the following hold. \begin{itemize}
        \item $\pi$ is surjective and universally closed.
        \item $\pi$ induces a bijection of closed points.
        \item $\pi$ is universal for maps to algebraic spaces.
        \item For every geometric point $x: \Spec \overline{k} \to \sM$ with closed image, the stabilizer group $G_x$ is linearly reductive.
        \item If $\sM$ is locally Noetherian, then so is $M$, and $\pi_*$ preserves coherence.
        \item If $\sM$ is of finite type over a field, then so is $M$.
    \end{itemize}
\end{prop}

We recall the following criterion \cite[Theorem 10.3]{AlperGMS} for a locally free sheaf $\sF$ on $\sM$ to descend to the good moduli space $M$. 
\begin{prop}\label{vbtogms}
    If $\pi: \sM \to M$ is a good moduli space and $\sM$ is locally Noetherian, then the pullback morphism $\pi^*: \Coh(M) \to \Coh(\sM)$ induces an equivalence of categories between locally free sheaves on $M$ and those locally free sheaves $\sF$ on $\sM$ such that for every geometric point $x: \Spec k \to \sM$ with closed image, the induced representation $x^*\sF$ of the stabilizer $G_x$ is trivial.
\end{prop}

The existence of a good moduli space for a given algebraic stack is a subtle question. One answer is given in \cite{AHLH}, where for a large class of stacks the authors give necessary and sufficient conditions for existence of a good moduli space in terms of certain valuative criteria. As an application, the authors construct proper good moduli spaces for various moduli stacks $\sM^{\mathrm{ss}}_\sA$ parameterizing objects in an abelian category $\sA$ that are semistable with respect to a rather general notion of stability on $\sA$. This construction includes stacks of Bridgeland semistable objects on a smooth, projective variety $X$ with respect to a numerical stability condition $\si = (\sA, Z) \in \Stab(X)$ whose heart $\sA$ is Noetherian and satisfies the ``generic flatness property'', and for which the moduli stacks $\sM^\si(v)$ are of finite type. See \cite[Section 7]{AHLH} for details, especially Theorem 7.25 and Example 7.27. In particular, we have the following.
\begin{thm}\label{gmsexists}
    Let $X$ be a smooth, projective surface over $\C$, let $v \in \Kn(X)$ be a numerical class, and let $\si \in \Stab(X)$ be a stability condition constructed by tilting with respect to slope-stability as in Section \ref{section:stabcondsurf}. The moduli stack $\sM^\si(v)$ of $\si$-semistable objects of class $v$ admits a good moduli space map $\sM^\si(v) \to M^\si(v)$, where $M^\si(v)$ is a proper algebraic space over $\C$. The closed points of $M^\si(v)$ are in bijection with S-equivalence classes of $\si$-semistable objects of class $v$.
\end{thm}


\section{Determinantal line bundles}
In this section we recall the construction of determinantal line bundles and consider some particular line bundles on a Bridgeland moduli space arising from this construction. Material for this section follows \cite[\href{https://stacks.math.columbia.edu/tag/0FJI}{Tag 0FJI}]{stacks-project}, \cite[\href{https://stacks.math.columbia.edu/tag/0FJW}{Tag 0FJW}]{stacks-project}, and \cite[Section 8.1]{HL}. The original exposition is \cite{KM76}.

\subsection{Construction of determinantal line bundles}
Let $S$ be a scheme. The rule that sends a locally free sheaf $F$ to its determinant line bundle $\det(F) = \bigwedge^{\rk(F)} F$ extends to a functor
\[ \det: \{\mathrm{perfect\, complexes\, on\,} S\} \to \{\mathrm{invertible\,} \mathrm{sheaves\, on}\, S \}. \]
Moreover, for any short exact sequence
\[ 0 \to F' \to F \to F'' \to 0 \]
of locally free sheaves, there is a canonical isomorphism $\det(F) \to \det(F') \otimes \det(F'')$, so in particular we obtain an induced homomorphism of abelian groups
\[ \det: K_0(S) \to \Pic(S), \]
where $K_0(S)$ denotes the Grothendieck group of vector bundles on $S$. These constructions commute with pullbacks in the sense that if $\pi: S' \to S$ is a morphism of schemes and $F$ is a locally free sheaf or a perfect complex on $S$, then canonically $\det(\pi^* F) \cong \pi^*\det(F)$.

Let now $X$ be a smooth, proper variety over $\C$, $S$ a scheme of finite type over $\C$, and $\sE \in D^b(S \times X)$ a perfect complex. Note that since $X$ is smooth, we have $K_0(X) \cong K(X)$. Consider the diagram:
\begin{center}
    \begin{tikzpicture}
    \matrix (m) [matrix of math nodes, row sep=1em, column sep=1em]
    { & S \times X & \\
    S & & X \\};
    \path[->] 
    (m-1-2) edge node[auto,swap] {$ p $} (m-2-1)
    (m-1-2) edge node[auto] {$ q $} (m-2-3)
    ;
    \end{tikzpicture}
\end{center}
Any coherent sheaf $F$ on $X$ is perfect as an object of $D^b(X)$, and hence so is the complex $\sE \otimes q^*F$ on $S \times X$. Thus, by \cite[\href{https://stacks.math.columbia.edu/tag/0B91}{Tag 0B91}]{stacks-project}, the derived pushforward $R p_* (\sE \otimes q^* F)$ is perfect on $S$. Composing with the determinant map gives a homomorphism of abelian groups
\[ \la_\sE: K(X) \to \Pic(S), \quad [F] \mapsto \det R p_*(\sE \otimes q^* F) \]
called the {\bf Donaldson morphism}. Moreover, since the formation of the pushforward $R p_*(\sE \otimes q^* F)$ commutes with base change, so does the formation of $\la_\sE$ the sense that if $\pi: S' \to S$ is a morphism of schemes, then the composition
\[ K(X) \xrightarrow{\la_\sE} \Pic(S) \xrightarrow{\pi^*} \Pic(S') \]
equals $\la_{(\pi \times \id_X)^* \sE}$. Basic properties of the map $\la_\sE$ are listed in \cite[Lemma 8.1.2]{HL}.

The Donaldson morphism respects numerical equivalence and hence a induces homomorphism
\[ \la_\sE: \Kn(X) \to \Num(S), \]
where $\Num(S)$ denotes $\Pic(S)$ modulo numerical equivalence, and furthermore extends to a linear map
\[ \la_\sE: \Kn(X)_\R \to \Num(S)_\R \]
of real vector spaces, where $\Kn(X)_\R = \Kn(X) \otimes \R$, and $\Num(S)_\R = \Num(S) \otimes \R$ is the group of real divisor classes.

This construction readily generalizes to algebraic stacks, and in particular, the Donaldson morphism lets us construct line bundles on Bridgeland moduli stacks. Let $X$ be a smooth, projective variety, let $v \in \Kn(X)$ be a numerical class, and let $\si \in \Stab(X)$ be a stability condition such that the moduli stack $\sM^\si(v)$ is algebraic. Let $\sE \in D^b(\sM^\si(v) \times X)$ be an $\sM^\si(v)$-perfect object (for example $\sE$ could be the universal complex), and consider the diagram:
\begin{center}
    \begin{tikzpicture}
    \matrix (m) [matrix of math nodes, row sep=1em, column sep=1em]
    { & \sM^\si(v) \times X & \\
    \sM^\si(v) & & X \\};
    \path[->] 
    (m-1-2) edge node[auto,swap] {$ p $} (m-2-1)
    (m-1-2) edge node[auto] {$ q $} (m-2-3)
    ;
    \end{tikzpicture}
\end{center}
If $F$ is a coherent sheaf on $X$, we obtain the line bundle
\[ \la_\sE(F) \coloneqq \det(R p_*(\sE \otimes q^*F)) \]
on $\sM^\si(v)$, and this induces a group homomorphism
\[ \la_\sE: K(X) \to \Pic(\sM^\si(v)). \]

\subsection{Sections of determinantal line bundles}
In special situations, the above construction of a determinantal line bundle also yields a canonical section of the line bundle. Namely, if $E \in D^b(T)$ is a perfect complex of rank 0 whose cohomology sheaves $\sH^i(E)$ vanish whenever $i \neq -1,0$, then Zariski locally on $T$ the complex $E$ can be represented by a 2-term complex of locally free sheaves
\[ \cdots \to 0 \to E_{-1} \xrightarrow{f} E_0 \to 0 \to \cdots, \quad \rk(E_{-1}) = \rk(E_0). \]
The map $f$ induces a section $\det(f): \Oh_T \to \det(E_0) \otimes \det(E_{-1})^\vee$, and these local section glue to a global section $\de_E \in \Ga(T, \det(E))$. Moreover, the formation of this section commutes with pullbacks, in the sense that if $\pi: T' \to T$ is a morphism of schemes, then the sections $\de_{\pi^*E}$ and $\pi^*\de_E$ are identified under the canonical isomorphism $\det(\pi^* E) \cong \pi^*\det(E)$. See \cite[\href{https://stacks.math.columbia.edu/tag/0FJX}{Tag 0FJX}]{stacks-project}.

Let now $X$ be a smooth, projective variety, $S$ a finite type scheme or algebraic stack, and $\sE \in D^b(S \times X)$ an $S$-perfect family of objects of class $v \in \Kn(X)$. For a $\C$-point $t \in S$, denote by $\sE_t$ the restriction of $\sE$ to the fiber $\{t\} \times X$ over $t$. The following lemma gives a criterion for when the line bundle $\la_{\sE}(F)^\vee$ on $S$ has a section for some $F \in \Coh(X)$, and when this section is nonzero at some $t \in S$. Denote by $\Hh^i(X, -) \coloneqq H^i(R\Ga(X, -)), i \in \Z$ the hypercohomology functors on $D^b(X)$.
\begin{lem}\label{detsection}
    Let $X$ be a smooth, projective variety and $S$ a scheme or an algebraic stack of finite type over $\C$. Let $\sE \in D^b(S \times X)$ be an $S$-perfect family of objects of class $v \in \Kn(X)$, and let $F$ be a locally free sheaf on $X$.
    \begin{enumerate}[(a)]
        \item If for all $\C$-points $t \in S$, we have $\Hh^i(X, \sE_t \otimes F) = 0$ whenever $i \neq 0, 1$, and 
        \[ \chi(X, \sE_t \otimes F) = \dim \Hh^0(X, \sE_t \otimes F) - \dim \Hh^1(X, \sE_t \otimes F) = 0, \]
        then the line bundle $\la_\sE(F)^\vee$ on $S$ has a canonical section $\de_F$.
        \item In addition, if for some $t \in S$ we have 
        \[ \Hh^0(X, \sE_t \otimes F) = \Hh^1(X, \sE_t \otimes F) = 0, \]
        then the section $\de_F$ is nonzero at $t$.
    \end{enumerate}
\end{lem}
\begin{proof}
    Cohomology and Base Change implies that $R^i p_*(\sE \otimes q^* F) = 0$ for $i \neq 0,1$, and thus locally on $S$, the object $R p_*(\sE \otimes q^*F)$ can be represented by a complex
    \[ \cdots \to 0 \to \sG_0 \xrightarrow{f} \sG_1 \to 0 \to \cdots. \]
    By \cite[\href{https://stacks.math.columbia.edu/tag/0B91}{Tag 0B91}]{stacks-project}, the formation of $R p_*(\sE \otimes q^*F)$ commutes with base change, and so for any $t \in S$, we have 
    \[ \sum_i (-1)^i \rk(\sG_i) = \sum_i (-1)^i \rk(\sG_i|_t) = \sum_i (-1) \dim \Hh^i(X, \sE_t \otimes F) = 0, \]
    so $\rk \sG_0 = \rk \sG_1$.
    
    Moreover, if $\Hh^1(X, \sE_t \otimes F) = 0$, then by Cohomology and Base Change $R^1 p_*(\sE \otimes q^*F) = 0$ in a neighborhood of $t$, so the map $f: \sE_0 \to \sE_1$ is surjective in a neighborhood of $t$, hence an isomorphism, and so its determinant is nonzero at $t$.
\end{proof}

\subsection{Stabilizer action at closed points}
We now compute how the stabilizer of a polystable object in $\sM^\si(v)$ acts on a determinantal line bundle $\la_\sE(F)$. We will use this in the next section to show that certain natural line bundles on $\sM^\si(v)$ descend to the good moduli space.

More generally, let $X$ be a smooth, projective variety, and let $E$ be a direct sum of simple objects in $D^b(X)$. If $F$ is a locally free sheaf on $X$, we want to know how $g \in \Aut(E)$ acts on the 1-dimensional vector space
\[ \det R \Ga(X, E \otimes F) = \bigotimes_{i \in \Z} (\det H^i(X, E \otimes F))^{(-1)^i}. \]
First consider the case $E \cong S^{\oplus r}$ where $S$ is a simple object and $r \ge 1$, so that $\Aut(E) \cong \GL_r$, and we can view an element $g \in \Aut(E)$ as an invertible matrix $g = (g_{kl})$. Thus, $g$ acts on $H^i(X, E \otimes F) \cong H^i(X, S \otimes F)^{\oplus r}$ by a block diagonal matrix consisting of $\dim H^i(X, S \otimes F)$ diagonal copies of $g$, and hence on
\[ \det H^i(X, E \otimes F) \cong (\det H^i(X, S \otimes F))^{\otimes r} \]
by multiplication by $\det(g)^{\dim H^i(X, S \otimes F)}$, and on $\det R \Ga(X, E \otimes F)$ by
\[ \prod_{i=0}^n \left((\det(g)^{\dim H^i(X, S \otimes F)}\right)^{(-1)^i} = \det(g)^{\chi(X, S \otimes F)}. \]
Next, consider the case $E \cong S_1^{\oplus r_1} \oplus \cdots \oplus S_m^{\oplus r_m}$ for mutually non-isomorphic simple objects $S_1, \ldots, S_m \in D^b(X)$. Now
\[ \Aut(E) \cong \GL_{r_1} \times \cdots \times \GL_{r_m}. \]
An element $g = (g_1, \ldots, g_m) \in \Aut(E)$ acts on
\[ H^i(X, E \otimes F) \cong H^i(X, S_1 \otimes F)^{\oplus r_1} \oplus \cdots \oplus H^i(X, S_m \otimes F)^{\oplus r_m} \]
by a block diagonal matrix with the matrix $g_j$ on the diagonal 
\[ \dim H^i(X, S_j \otimes F) \]
times. Thus, $g$ acts on $\det H^i(X, E \otimes F)$ by multiplication by
\[ \det(g_1)^{\dim H^i(X, S_1 \otimes F)} \cdots \det(g_m)^{\dim H^i(X, S_m \otimes F)}, \]
and hence on $\det R \Ga(X, E \otimes F)$ by
\[ \det(g_1)^{\chi(X, S_1 \otimes F)} \cdots \det(g_m)^{\chi(X, S_m \otimes F)}. \]
The same analysis extends to the case where we replace $F$ by an element $u \in K(X)$, and so we have the following.
\begin{prop}\label{lbtogms}
    Let $\sM^\si(v)$ be the stack of $\si$-semistable objects of class $v$, and let $\sE$ denote the universal complex on $\sM^\si(v) \times X$. If $u \in K(X)$ is any class, and $E \cong S_1^{\oplus r_1} \oplus \cdots \oplus S_m^{\oplus r_m}$ is a $\si$-polystable object corresponding to a closed point of $t \in \sM^\si(v)$, an element $g = (g_1,\ldots,g_m) \in \Aut(E)$ acts on the fiber of $\la_\sE(u)$ on $\sM^\si(v)$ at $t$ by multiplication by
    \[ \det(g_1)^{\chi(X, S_1 \otimes u)} \cdots \det(g_m)^{\chi(X, S_m \otimes u)}. \]
    In particular, if $\chi(X, S_i \otimes u) = 0$ for each $i$, then $\Aut(E)$ acts trivially on the fiber.
\end{prop}


\section{The nef line bundle}
In \cite{BM}, the authors construct a natural numerical class of line bundles with strong positivity properties on a Bridgeland moduli space. 

\subsection{Definition and positivity properties}\label{positiveLBdefandprops}
Let $(X,H)$ be a smooth, projective, polarized surface, let $\si = (\sA, Z) \in \Stab(X)$ be a stability condition, and let $v \in \Kn(X)$ be a numerical class. Assume that the moduli stack $\sM^\si(v)$ of $\si$-semistable objects in $\sA$ of class $v$ is algebraic, and denote by $\sE$ the universal complex on $\sM^\si(v) \times X$. Consider the following diagram:
\begin{center}
    \begin{tikzpicture}
    \matrix (m) [matrix of math nodes, row sep=1em, column sep=1em]
    { & \sM^\si(v) \times X & \\
    \sM^\si(v) & & X \\};
    \path[->] 
    (m-1-2) edge node[auto,swap] {$ p $} (m-2-1)
    (m-1-2) edge node[auto] {$ q $} (m-2-3)
    ;
    \end{tikzpicture}
\end{center}
The Donaldson morphism
\[ \la_\sE: K_0(X) \to \Pic(\sM^\si(v)), \quad [F] \mapsto \det R p_*(\sE \otimes q^* F) \]
from the previous section induces a map 
\[ \la_\sE: \Kn(X)_\R \to \Num(\sM^\si(v))_\R. \]
Define a real divisor class on $\sM^\si(v)$ by applying the Donaldson morphism to the unique class $w_Z \in \Kn(X)_\R$ determined by the condition
\[ \chi(w_Z, -) = \im\left(-\frac{Z(-)}{Z(v)}\right). \]
This condition indeed defines a unique class since the Euler pairing $\chi(-,-)$ induces a perfect pairing on $\Kn(X)_\R$. Denote this numerical class by $\sL_\si \coloneqq \la_\sE(w_Z)$. The following is \cite[Lemma 3.3]{BM}, and it is the main result of the paper.
\begin{thm}\label{BMpositivity}
    Let $C$ be a projective, integral curve over $\C$, and let $f: C \to \sM^\si(v)$ be a morphism.
    \begin{enumerate}[(1)]
        \item $\deg f^* \sL_\si \ge 0$.
        \item If $\deg f^* \sL_\si = 0$, then for any two closed points $s, t \in C$, the objects \\ $(f \times \id_X)^*\sE|_{\{s\} \times X}$ and $(f\times \id_X)^*\sE|_{\{t\} \times X}$ are S-equivalent.
    \end{enumerate}
\end{thm}
In \cite{BM}, part (2) is stated so that the objects $(f \times \id_X)^*\sE|_{\{t\} \times X}$ are S-equivalent for points $t$ in some nonempty open subscheme $U \subs C$. We deduce the above statement from Theorem \ref{gmsexists} as follows. If $\sM^\si(v) \to M^\si(v)$ denotes the good moduli space map, then the composition $C \to \sM^\si(v) \to M^\si(v)$ maps the dense open $U$ to a point, hence must be constant, and so the objects $(f\times \id_X)^*\sE|_{\{t\} \times X}$ are all S-equivalent. 

We would like to know that the real divisor class $\sL_\si$ descends to the good moduli space $M^\si(v)$. In for instance \cite{BM} and \cite{nuer} this is done using a so-called quasi-universal family on the stable locus of the moduli space. However, we can achieve this on all of $M^\si(v)$ as follows.
\begin{lem}\label{nefdescendtogms}
    If $w \in K(X)$ is a class whose image in $\Kn(X)_\R$ is a multiple of $w_Z$, then $\la_\sE(w)$ descends to the good moduli space $M^\si(v)$.
\end{lem} 
\begin{proof}
    Write $w = b w_Z$ in $\Kn(X)_\R$ with $b \in \R$. We check the condition in Proposition \ref{lbtogms}. If $E \cong E_1^{\oplus r_1} \oplus \cdots \oplus E_m^{\oplus r_m}$ is a $\si$-polystable object of class $v$, then for each $i$, the complex number $Z(E_i)$ lies on the same ray as $Z(E)$, so that $Z(E_i)/Z(v)$ is real, and so
    \[ \chi(w, E_i) = b\,\chi(w_Z, E_i) = b\,\im\left(-\frac{Z(E_i)}{Z(v)}\right) = 0. \]
    Thus, $\la_\sE(w)$ descends to a line bundle on the good moduli space.    
\end{proof}
This lets us define a numerical class $L_\si$ on $M^\si(v)$ by setting $L_\si = \frac{1}{b}[\sN]$ where $\pi^*\sN = \la_\sE(w)$ such that $w \in K(X)$ satisfies $w \equiv b w_Z$ in $\Kn(X)$ and $b > 0$. The class $L_\si$ is independent of the choice of $w$. We would like to know that $L_\si$ enjoys the same positivity properties as $\sL_\si$.
\begin{lem}\label{gmsnef}
    Let $C$ be a smooth, projective, integral curve over $\C$, and let $g: C \to M^\si(v)$ be a morphism.
    \begin{enumerate}[(1')]
        \item $\deg g^* L_\si \ge 0$.
        \item If $\deg g^* L_\si = 0$, then $g$ is constant.
    \end{enumerate}
\end{lem}
\begin{proof}
    By Lemma \ref{finitecurveextension} below, we can find a commutative diagram
    \begin{center}
    \begin{tikzpicture}
    \matrix (m) [matrix of math nodes, row sep=2em, column sep=2em]
    { C' & \sM^\si(v)  \\
    C & M^\si(v) \\};
    \path[->] 
    (m-1-1) edge node[auto] {$ f $} (m-1-2)
    (m-1-1) edge node[auto,swap] {$ \phi $} (m-2-1)
    (m-1-2) edge node[auto] {$ \pi $} (m-2-2)
    (m-2-1) edge node[auto,swap] {$ g $} (m-2-2)
    ;        
    \end{tikzpicture}
    \end{center}
    where $C'$ is a smooth, projective curve, $\phi$ is a finite morphism, and $\pi$ is the good moduli space map. To prove (1') we note that
    \[ \deg(\phi) \deg g^* L_\si = \deg (g \circ \phi)^* L_\si = \deg(\pi \circ f)^* L_\si = \deg f^*\sL_\si \ge 0, \]
    so since $\deg(\phi) > 0$, we get $\deg g^* L_\si \ge 0$.
    
    To prove (2'), assume that $\deg g^*L_\si = 0$. Then also $f^* \sL_\si = 0$, so by part (2) of Theorem \ref{BMpositivity}, the family parameterized by $C'$ consists of S-equivalent objects, so the composition $\pi \circ f: C' \to M^\si(v)$ maps every closed point of $C'$ to the same point $p_0 \in |M^\si(v)|$, and the same holds for $g: C \to M^\si(v)$, and so the scheme-theoretic image of $g$ is a closed point of $M^\si(v)$.
\end{proof}

\begin{lem}\label{finitecurveextension}
    Let $\sM$ be an algebraic stack of finite type that admits a good moduli space $\pi: \sM \to M$ with $M$ proper. Let $C$ be a smooth, proper curve, and let $g: C \to M$ be a morphism. There exists a commutative diagram
    \begin{center}
    \begin{tikzpicture}
    \matrix (m) [matrix of math nodes, row sep=2em, column sep=2em]
    { C' & \sM  \\
    C & M \\};
    \path[->] 
    (m-1-1) edge node[auto] {$ f $} (m-1-2)
    (m-1-1) edge node[auto,swap] {$ \phi $} (m-2-1)
    (m-1-2) edge node[auto] {$ \pi $} (m-2-2)
    (m-2-1) edge node[auto,swap] {$ g $} (m-2-2)
    ;        
    \end{tikzpicture}
    \end{center}
    where $C'$ is smooth and proper and $\phi$ is finite.
\end{lem}
\begin{proof}
    Let $U \to \sM$ be a smooth surjection with $U$ a scheme of finite type. The fiber product $C \times_M U$ is also a scheme of finite type and $\pi_C: C \times_M U \to C$ is surjective. The scheme-theoretic fiber $\pi_C^{-1}(\eta)$ over the generic point $\eta \in C$ is of finite type over the function field of $C$, hence contains a closed point $\tau \in \pi_C^{-1}(\eta)$ whose residue field $K = \ka(\tau)$ is a finite extension of the function field $K(C) = \ka(\eta)$. Let $C_\tau$ denote the normalization of $C$ in the field $\ka(\tau)$. Since $U$ is of finite type, we can extend the map $\Spec K \to U$ over an open subscheme $V \subs C_\tau$ and obtain a commutative diagram
    \begin{center}
    \begin{tikzpicture}
    \matrix (m) [matrix of math nodes, row sep=2em, column sep=2em]
    { \Spec K & V & U & \sM \\
    & C_\tau & C & M \\};
    \path[right hook->] 
    (m-1-1) edge node[auto] {$ $} (m-1-2)
    (m-1-2) edge node[auto] {$ $} (m-2-2)
    ;
    \path[->]
    (m-1-2) edge node[auto] {$ $} (m-1-3)
    (m-1-3) edge node[auto] {$ $} (m-1-4)
    (m-1-4) edge node[auto] {$ \pi $} (m-2-4)
    (m-2-2) edge node[auto,swap] {$ $} (m-2-3)
    (m-2-3) edge node[auto] {$ g $} (m-2-4)
    ;        
    \end{tikzpicture}
    \end{center}
    By applying \cite[Theorem A.8]{AHLH} to the local rings of the finitely many points in the complement $C_\tau \setminus V$, we find a finite extension $K'$ of $K$ such that the normalization $C'$ of $C_\tau$ in $K'$ admits a map $C' \to \sM$.
\end{proof}

\subsection{The nef line bundle on the vertical wall}
We can describe the class $w_Z \in \Kn(X)$ more explicitly in the case of the stability condition 
\[ \si = (\sA, Z) = (\Coh^{\be_0}(X)[-1], -Z_{\al, \be_0}), \; \mathrm{where} \; \be_0 = \frac{H \cdot \ch_B^1(v)}{H^2 \ch_B^0}, \; \al > 0, \] that is, when $\si$ lies on the vertical wall for $v$. Recall that $H \subs X$ denotes a fixed very ample divisor. Denote $h = [\Oh_H] \in K(X)$.
\begin{propdef}\label{nefclassonverticalwall}
    Let $(X,H)$ be a smooth, projective, polarized surface, let $v \in \Kn(X)$ be a class of positive rank, and let $\si \in \Stab(X)$ lie on the vertical wall for $v$ as in Section \ref{section:stabvertwall}. Define 
    \[ u = -\chi(v \cdot h^2) h + \chi(v \cdot h) h^2 \in K(X). \]
    We have 
    \[ w_Z = -\frac{\al}{\rk(v) Z(v) \deg X} u. \]
\end{propdef}
Note that $Z(v)$ is a negative real number, so $w_Z$ is indeed a positive multiple of $u$. Thus, $\la_\sE(u)$ enjoys the same positivity properties as $\sL_\si = \la_\sE(w_Z)$.
\begin{proof}
    Fix a closed point $p \in X$, and denote the numerical Todd class of $X$ by
    \[ \td_X = 1 - \frac{1}{2} K_X + \chi(\Oh_X) [p]. \]
    We will use the Hirzebruch-Riemann-Roch formula:
    \[ \chi(a \cdot b) = \int_X \ch(a) \ch(b) \td_X \]
    for $a, b \in \Kn(X)$.
    
    Let $a \in K(X)$ be arbitrary. To compute $\chi(a \cdot u)$, we may replace $u$ with something numerically equivalent. Now $h^2 \equiv \deg(X) [\Oh_p]$ by Bertini, and $v \cdot [\Oh_p] = \rk(v)$, so if we set
    \[ u' = - \rk(v) h + \chi(v \cdot h) [\Oh_p] \]
    we can consider the class $\deg(X) u'$ instead of $u$. Moreover, 
    \[ \ch(\Oh_H) = \ch(\Oh_X) - \ch(\Oh_X(-H)) = H - \frac{1}{2} H^2, \quad \ch(\Oh_p) = [p], \]
    so by Hirzebruch-Riemann-Roch,
    \begin{align*}
        \chi(v \cdot h) & = \int_X (\rk(v) + \ch_1(v))\left(H - \frac{1}{2}H^2\right)\left(1 - \frac{1}{2}K_X\right) \\
        & = \ch_1(v) \cdot H - \frac{\rk(v)}{2} H\cdot(H + K_X),
    \end{align*}
    and hence
    \begin{align*}
        \ch(u') & = -\rk(v) \ch(\Oh_H) + \chi(v \cdot [\Oh_H]) \ch(\Oh_p) \\
        & = -\rk(v)\left(H - \frac{1}{2}H^2\right) + \left(\ch_1(v) \cdot H - \frac{\rk(v)}{2} H(H + K_X)\right)[p] \\
        & = -\rk(v) H + \left(\ch_1(v) \cdot H - \frac{\rk(v)}{2} H \cdot K_X\right)[p]
    \end{align*}
    Since $\ch(u')$ is in $A^{\ge 1}(X)$, we only need to know the $A^{\le 1}(X)$-part of $\ch(a)\td_X$, which is
    \[ \ch(a) \td_X = (\ch_0(a) + \ch_1(a)) (1 - \frac{1}{2} K_X) \equiv \ch_0(a) - \frac{1}{2} \ch_0(a) K_X + \ch_1(a). \]
    Putting everything together, we now calculate
    \begin{align*}
        \chi(a \cdot u') & = \int_X \ch(a) \ch(u') \td_X \\
        & = \int_X \left(\ch_0(a) - \frac{1}{2} \ch_0(a) K_X + \ch_1(a)\right)\cdot  \\
        & \quad\left(-\ch_0(v) H + \ch_1(v) H - \frac{1}{2} \ch_0(v) H \cdot K_X\right) \\
        & = H \cdot (\ch_0(a) \ch_1(v) - \ch_0(v) \ch_1(a)) \\
        & = H \cdot (\ch^B_0(a) \ch^B_1(v) - \ch^B_0(v) \ch^B_1(a))
    \end{align*}
    Thus,
    \[ \chi(a \cdot u) = \deg X \chi(a \cdot u') = \deg X \cdot H (\ch^B_0(a) \ch^B_1(v) - \ch^B_0(v) \ch^B_1(a)). \]
    
    Next we calculate $\im(-Z(a)/Z(v))$. Recall that since $\si$ is on the vertical wall for $v$, the quantity $Z(v)$ is a negative real number, and so 
    \[ \im\left(-\frac{Z(a)}{Z(v)}\right) = -\frac{1}{Z(v)} \im Z(a). \]
    Now 
    \begin{align*}
        \im Z(a) & = \al(\be_0 H^2 \ch^B_0(a) - H \cdot \ch^B_1(a)) \\
        & = \al \left( \frac{H \cdot \ch^B_1(v)}{H^2 \ch^B_0(v)} H^2 \ch^B_0(a) - H \cdot \ch^B_1(a) \right) \\
        & = \frac{\al}{\ch^B_0(v)} H (\ch^B_0(a) \ch^B_1(v) - \ch^B_0(v) \ch^B_1(u)).
    \end{align*}
    Thus, we see that
    \[ \im\left(-\frac{Z(a)}{Z(v)}\right) = -\frac{\al}{\rk(v) Z(v) \deg X} \chi(a \cdot u). \]
\end{proof}

\section{Projectivity of the good moduli space}

Let $(X, H)$ be a smooth, projective, polarized surface over $\C$, let $v \in \Kn(X)$ be a numerical class with $\rk(v) > 0$, and let $\si$ be a stability condition lying on the vertical wall for $\si$ considered in Section \ref{section:stabvertwall}, that is,
\[ \si = (\sA, Z) = (\Coh^{\be_0}(X)[-1], -Z_{\al, \be_0}), \; \mathrm{where} \; \be_0 = \frac{H \cdot \ch_B^1(v)}{H^2 \ch_B^0}, \; \al > 0. \]
Let $\sM^\si(v)$ be the stack of $\si$-semistable objects of class $v$ in $\sA$, and let $\sE$ be the universal complex on $\sM^\si(v) \times X$. In this section we prove that the good moduli space $M^\si(v)$ of the stack $\sM^\si(v)$ is projective. 

Recall from Section \ref{positiveLBdefandprops} that $\sL_\si$ denotes the natural nef line bundle on $\sM^\si(v)$ and $L_\si$ the corresponding line bundle on $M^\si(v)$. We will show that $L_\si$ is ample. Since by Lemma \ref{nefdescendtogms}, $L_\si$ is strictly positive on any proper curve in $M^\si(v)$, it suffices to show that $L_\si$ is semiample, meaning that some tensor power is globally generated. Moreover, since by Proposition \ref{vbtogms}, sections of $\sL_\si$ descend to sections of $L_\si$, it suffices to show that $\sL_\si$ is semiample. To produce sections of $\sL_\si$, we expand on techniques used in \cite{li} for constructing a scheme structure on the Uhlenbeck compactification, and in \cite{seshadri} for constructing moduli spaces of vector bundles on a curve. 

The idea is as follows. First, we explain how to obtain a diagram
\begin{center}
    \begin{tikzpicture}
    \matrix (m) [matrix of math nodes, row sep=3em, column sep=0.5em]
    { & \sM^\si(v) \times X & & \sM^\si(v) \times C & \\
    \sM^\si(v) & & X & & C \\};
    \path[left hook->] 
    (m-1-4) edge node[auto,swap] {$ j $} (m-1-2)
    (m-2-5) edge node[auto] {$ i $} (m-2-3)
    ;
    \path[->]
    (m-1-2) edge node[auto,swap] {$ p $} (m-2-1)
    (m-1-2) edge node[pos=0.9,yshift=10pt] {$ q $} (m-2-3)
    (m-1-4) edge node[pos=0.7,yshift=-8pt] {$ p_C $} (m-2-1)
    (m-1-4) edge node[auto] {$ q_C $} (m-2-5)
    ;        
    \end{tikzpicture}
\end{center}
where $C$ is a smooth curve in the linear system $|a H|$ for $a > 0$, together with a locally free sheaf $G$ on $C$, with the property that the determinantal line bundle
\[ \la_{j^* \sE}(G)^\vee = \det(R p_{C*}(j^*\sE \otimes q_C^*G))^\vee \]
is a positive multiple of $\sL_\si$ on $\sM^\si(v)$. 

Next, by analyzing restrictions of $\si$-semistable objects $E$ to $C$, we apply Lemma \ref{detsection}(a) to show that $\la_{j^* \sE}(G)^\vee$ has a canonical global section $\de_G$ on $\sM^\si(v)$. Moreover, we show that for a given $\C$-point $t \in \sM^\si(v)$, we can choose the curve $C$ and the sheaf $G$ so that the section $\de_G$ is nonvanishing at $t$. To do this, recall from Proposition \ref{ss-object-vertical-classification} that the complex $\sE_t$ on $\{t\} \times X = X$ fits in an exact triangle
\[ F \to \sE_t \to T[-1] \]
in $D^b(X)$, where $F$ is a $\mu$-semistable torsion-free sheaf and $T$ is a torsion sheaf with 0-dimensional support. Using a restriction theorem for $\mu$-stability, we show that we can choose $a \gg 0$ and $C \in |a H|$ so that
\begin{enumerate}[(1)]
    \item $C$ avoids the support of $T$, and
    \item the restriction $\sE_t|_C = F|_C$ to $C$ is slope-semistable.
\end{enumerate}
Using a characterization of semistability on a curve due to Faltings and Seshadri, we find a locally free sheaf $G$ on $C$ with the property that 
\[ H^0(C, F|_C \otimes G) = H^1(C, F|_C \otimes G) = 0, \]
and apply Lemma \ref{detsection}(b) to translate this into the nonvanishing of $\de_G$ at $t$.

Finally, by varying $C$ and $G$, we produce a generating set of sections of some power of $\sL_\si$, or equivalently $L_\si$, and use Theorem \ref{gmsexists} and Lemma \ref{gmsnef} to show that the morphism $M^\si(v) \to \p^N$ induced by the sections is finite. From this we conclude that $M^\si(v)$ is projective.

\subsection{Sheaves on curves and the nef line bundle}
We begin to carry out the plan outlined above. To set up some notation, let $C \subs X$ be a smooth, connected curve in the linear system $|a H|$ for some $a > 0$, and consider the diagram:
\begin{center}
    \begin{tikzpicture}
    \matrix (m) [matrix of math nodes, row sep=3em, column sep=0.5em]
    { & \sM^\si(v) \times X & & \sM^\si(v) \times C & \\
    \sM^\si(v) & & X & & C \\};
    \path[left hook->] 
    (m-1-4) edge node[auto,swap] {$ j $} (m-1-2)
    (m-2-5) edge node[auto] {$ i $} (m-2-3)
    ;
    \path[->]
    (m-1-2) edge node[auto,swap] {$ p $} (m-2-1)
    (m-1-2) edge node[pos=0.9,yshift=10pt] {$ q $} (m-2-3)
    (m-1-4) edge node[pos=0.7,yshift=-8pt] {$ p_C $} (m-2-1)
    (m-1-4) edge node[auto] {$ q_C $} (m-2-5)
    ;        
    \end{tikzpicture}
\end{center}
where $p, q, p_C, q_C$ denote the projections, and $i$ and $j$ are closed embeddings. Let $\sE_C \coloneqq j^*\sE$ be the restriction of $\sE$ to $\sM^\si(v) \times C$, which is perfect relative to $\sM^\si(v)$. We have the Donaldson homomorphisms
\[ \la_\sE: K(X) \to \Pic(\sM^\si(v)), \quad \la_{\sE_C}: K(C) \to \Pic(\sM^\si(v)). \]
In addition, for any $n \in \Z$, we have a map $K(X) \to K(X), w \mapsto w(n)$ induced by the map on locally free sheaves $F \mapsto F(n) = F \otimes \Oh_X(n)$. Similarly we have a map $K(X) \to K(C), w \mapsto w|_C$ induced by $F \mapsto F|_C = i^*F$. We denote $h = [\Oh_H] \in K(X)$ as before.

Recall from Proposition \ref{nefclassonverticalwall} that for the class 
\[ u = -\chi(v \cdot h^2) h + \chi(v \cdot h) h^2 \in K(X), \]
the line bundle $\la_\sE(u)$ on $\sM^\si(v)$ is a positive multiple of the natural nef line bundle $\sL_\si$. We first establish the following.
\begin{propdef}\label{nefpowerfromcurves}
    Given an integer $a > 0$ and a smooth, connected curve $C \in |a H|$, define the class
    \[ w \coloneqq -\chi(v \cdot h \cdot [\Oh_C]) \cdot 1 + \chi(v \cdot [\Oh_C]) \cdot h \in K(X). \]
    We have an isomorphism
    \[ \la_{\sE_C}(w|_C) \cong \la_\sE(u)^{a^2}. \]
    Moreover, the class $-w|_C \in K(C)$ has positive rank, and so can be represented by a locally free sheaf $G$ on $C$.
\end{propdef}

The proof of the following simple lemma as well as part of the proof of Proposition \ref{nefpowerfromcurves} below are essentially included in the proof of \cite[Proposition 8.2.3]{HL}.

\begin{lem}\label{Ktheorylemma}
    If $C \in |a H|$ is a curve and $w' \in K(X)$ is arbitrary, then
    \[ \la_\sE(w' - w'(-a)) = \la_{\sE_C}(w'|_C). \]
\end{lem}
\begin{proof}
    Both sides of the equation are linear in $w'$, so it suffices to consider the class $w' = [F]$ for a locally free sheaf $F$ on $X$. On the one hand, we have
    \begin{align*}
        \la_{\sE_C}(F|_C) & = \det R p_{C*}(\sE_C \otimes q_C^* i^* F) \\
        & = \det R p_* j_*(\sE_C \otimes j^* q^* F) \\
        & = \det R p_* (j_*\sE_C \otimes q^* F). \qquad \mathrm{(projection\;formula)}
    \end{align*}
    On the other hand, pulling back the short exact sequence 
    \[ 0 \to \Oh_X(-a) \to \Oh_X \to j_* \Oh_C \to 0 \]
    along $q$, tensoring with $\sE \otimes q^* F$, and applying $R p_*$ gives the exact triangle
    \[ R p_* (\sE \otimes q^*(F \otimes \Oh_X(-a))) \to R p_* (\sE \otimes q^*F) \to R p_* (j_* \sE_C \otimes q^*F) \]
    in $D^b(\sM^\si(v))$, and so we obtain an isomorphism
    \[ \det R p_* (j_* \sE_C \otimes q^*F) \cong \la_\sE(F) \otimes \la_\sE(F(-a))^\vee = \la_\sE(w - w(-a)). \]
\end{proof}

\begin{proof}[Proof of Proposition \ref{nefpowerfromcurves}]
    By Lemma \ref{Ktheorylemma}, for the first statement it is enough to show that
    \[ w - w(-a) = a^2 u. \]
    Since $[\Oh_X(-1)] = 1 - h \in K(X)$, we have
    \[ [\Oh_X(-a)] = [\Oh_X(-1)]^a = (1 - h)^a = 1 - a h + \binom{a}{2} h^2, \]
    so the short exact sequence
    \[ 0 \to \Oh_X(-a) \to \Oh_X \to \Oh_C \to 0 \]
    gives $[\Oh_C] = 1 - [\Oh_X(-a)] = a h - \binom{a}{2} h^2$. In particular, $h \cdot [\Oh_C] = a h^2$. We now calculate
    \begin{align*}
        w - w(-a) & = w \cdot [\Oh_C] \\
        & = (-\chi(v \cdot h \cdot [\Oh_C]) \cdot 1 + \chi(v \cdot [\Oh_C]) \cdot h) (a h - \binom{a}{2} h^2) \\
        & = - \chi(v \cdot ah^2) (a h - \binom{a}{2} h^2) + \chi\left(v \cdot (a h - \binom{a}{2} h^2)\right) a h^2 \\
        & = a^2 ( - \chi(v \cdot h^2) h + \chi(v \cdot h) h^2) = a^2 u.
    \end{align*}
    For the second claim, we note that 
    \[ -w|_C = \chi(v|_C \cdot h|_C) \cdot 1 - \chi(v|_C) \cdot h|_C = a \rk(v) \deg X \cdot 1 - \chi(v|_C) \cdot h|_C, \]
    and so 
    \[ \rk(w|_C) = a \rk(v) \deg X \rk(1) - \chi(v|_C) \rk(h|_C) = a \rk(v) \deg(X) > 0. \] 
    Since $C$ is smooth, projective, and connected, the natural map 
    \[ K(C) \to \Z \oplus \Pic(C), \quad [F] \mapsto (\rk F, \det F) \] 
    is an isomorphism. Moreover, any class $(m, L) \in \Z \oplus K(C)$ with $m > 0$ can be represented by a locally free sheaf: take for instance $G = \Oh_C^{\oplus m-1} \oplus L$. In particular, there exist locally free sheaves $G$ of class $- w|_C$ on $C$.
\end{proof}

\subsection{Producing sections}
Our next task is to show that the construction of Proposition \ref{nefpowerfromcurves} yields a canonical section $\de_G$ of the line bundle $\la_{\sE_C}(G)^\vee \cong \la_\sE(u)^{m a^2}$ on $\sM^\si(v)$, and that by choosing $C$ and $G$ carefully, this section is nonzero at a given point $t \in \sM^\si(v)$. After some preparations, we prove this in Proposition \ref{globalgen}. In the proof, Lemmas \ref{restsingle} and \ref{hypercohovanishing} will be used to apply Lemma \ref{detsection}(a), and Lemmas \ref{restsemistable} and \ref{seshadrimainlemma} to apply Lemma \ref{detsection}(b).

\begin{lem}\label{restsingle} 
    Let $X$ be a smooth, projective surface and $C \subs X$ a smooth, projective curve. Assume $E \in D^b(X)$ fits in a triangle
    \[ F \to E \to T[-1], \]
    where $F$ is a torsion-free sheaf and $T$ is a torsion sheaf with 0-di\-men\-sion\-al support. The derived restriction $E|_C^\LL$ fits in a triangle
    \[ \sH^0(E|_C^\LL) \to E|_C^\LL \to \sH^1(E|_C^\LL)[-1] \]
    in $D^b(C)$, where $\sH^1(E|_C^\LL)$ is a torsion sheaf.
\end{lem}
\begin{proof}
Derived restriction to $C$ is a functor of triangulated categories, so we obtain a triangle
\[ F|_C^\LL \to E|_C^\LL \to T|_C^\LL[-1], \]
which yields a long exact sequence of cohomology sheaves
\[ \cdots \to \sH^i(F|_C^\LL) \to \sH^i(E|_C^\LL) \to \sH^i(T|_C^\LL[-1]) \to \sH^{i+1}(F|_C^\LL) \to \cdots. \]
To understand the terms in this sequence, we study the derived restrictions $F|_C^\LL$ and $T|_C^\LL$. Since pushforward of coherent sheaves along the inclusion $C \hookrightarrow X$ is exact, we may as well study the derived tensor products $F \otimes^\LL \Oh_C$ and $T \otimes^\LL \Oh_C$, where $\Oh_C$ is the structure sheaf of $C$ viewed as an $\Oh_X$-module. 

Since $C \subs X$ is a Cartier divisor, $\Oh_C$ has a resolution by line bundles
\[ 0 \to \Oh_X(-C) \xrightarrow{f} \Oh_X \to \Oh_C \to 0. \]
Thus, the objects $F \otimes^\LL \Oh_C$ and $T \otimes^\LL \Oh_C$ are represented by the complexes
\[ \sF_\bullet = [F(-C) \xrightarrow{f} F] \qquad \mathrm{and} \qquad \sT_\bullet = [T(-C) \xrightarrow{f} T] \]
respectively, where $F$ and $T$ are placed in degree 0. 

Since $F$ is by assumption torsion-free, the map $f: F(-C) \to F$ is injective. Thus, we see that $\sH^0(F|_C^\LL) = \sH^0(\sF_\bullet) = F|_C$ agrees with the ordinary restriction, and $\sH^i(F|_C^\LL) = \sH^i(\sF_\bullet) = 0$ for $i \neq 0$. Moreover, from $\sT_\bullet$ we see that $\sH^{-1}(T|_C^\LL) = \sH^{-1}(\sT_\bullet)$ is a subsheaf of $T(-C) \cong T$, and $\sH^0(T|_C^\LL) = \sH^0(\sT_\bullet)$ is a quotient of $T$, hence both are 0-dimensional, and $\sH^i(T|_C^\LL) = 0$ for $i \neq -1, 0$.

We now return to the triangle
\[ F|_C^\LL \to E|_C^\LL \to T|_C^\LL[-1] \]
at the beginning of the proof. Taking into account the shift in the last term, we obtain an exact sequence of cohomology sheaves
\[ 0 \to F|_C \to \sH^0(E|_C^\LL) \to \sH^{-1}(T|_C^\LL) \to 0 \to \sH^1(E|_C^\LL) \to \sH^0(T|_C^\LL) \to 0, \]
and also see that $\sH^i(E|_C^\LL) = 0$ if $i \neq 0,1$. In particular, $E|_C^\LL$ is supported in degrees $0$ and $1$, and $\sH^1(E|_C^\LL) \cong \sH^0(T|_C^\LL)$ is a torsion sheaf on $C$.
\end{proof}

\begin{lem}\label{hypercohovanishing}
    If $C$ is a projective curve and $E \in D^b(C)$ fits into a triangle
    \[ F \to E \to T[-1], \]
    with $F, T \in \Coh(C)$ and $T$ has 0-dimensional support, then the hypercohomology groups of $E$ satisfy
    \[ \Hh^i(C, E) = 0 \qquad \mathrm{for} \; i \neq 0, 1. \]
\end{lem}
\begin{proof}
    We have a long exact sequence of hypercohomology groups
    \[ \cdots \to \Hh^i(C, F) \to \Hh^i(C, E) \to \Hh^i(C, T[-1]) \to \Hh^{i+1}(C, F) \to \cdots. \]
    Since $\Hh^i(C, F) = 0$ whenever $i \neq 0, 1$, and 
    \[ \Hh^i(C, T[-1]) = \Hh^{i-1}(C, T) = 0 \] 
    whenever $i \neq 1$, the group $\Hh^i(C, E)$ can be nonzero only if $i = 0,1$.
\end{proof}

We pause to recall that $\si = (\sA, Z)$ denotes a stability condition on the vertical wall for the class $v \in \Kn(X)$, and that by Proposition \ref{ss-object-vertical-classification}, any $\si$-semistable object $E \in \sA$ of class $v$ fits in a triangle
\[ F \to E \to T[-1] \]
where $F$ is a $\mu$-semistable torsion-free sheaf and $T$ is a torsion sheaf with 0-dimensional support. 

\begin{lem}\label{restsemistable}
    If $E \in \sA \subs D^b(X)$ is a $\si$-semistable object of class $v$, there exists a smooth, projective, connected curve $C \subs X$ in the linear system $|a H|$ for $a \gg 0$ such that the derived restriction of $E$ to $C$ is a slope-semistable locally free sheaf.    
\end{lem}
\begin{proof} 
We use Flenner's restriction theorem, \cite[Theorem 7.1.1]{HL}. Specialized to the case at hand, it states the following. Let $a$ be an integer satisfying
\[ \frac{\binom{a+2}{a} - a - 1}{a} = \frac{a+1}{2} > \deg(X) \cdot \max\left\{\frac{r^2 - 1}{4}, 1\right\}. \]
If $F$ is a $\mu$-semistable sheaf of rank $r = \rk(v)$, then there is a nonempty open subset $U$ in the complete linear system $|aH|$, such that every $C \in U$ is smooth and the restriction of $F$ to $C$ is semistable. Note that since $F$ is torsion-free, its derived and ordinary restriction to $C$ agree.

Now let $E \in \sA$ be a $\si$-semistable object of class $v$ fitting in an exact triangle
\[ F \to E \to T[-1] \]
as above. For any smooth curve $C \subs X$, derived restriction to $C$ gives an exact triangle
\[ F|_C^\LL \to E|_C^\LL \to T|_C^\LL[-1]. \]
If $C$ does not pass through the finitely many closed points $p_1,\ldots,p_m$ in the support of $T$, then $T|_C^\LL = 0$, and thus $F|_C^\LL \cong E|_C^\LL$. Now for each point $p_i$, the subset in $|a H|$ of curves not passing through $p_i$ is a nonempty open subset $U_i \subs |a H|$. Since $|a H|$ is irreducible, the intersection $U' = U \cap U_1 \cap \cdots \cap U_m$ is also nonempty, and any curve $C \in U'$ has the desired property. 
\end{proof}

The following powerful result of Faltings and Seshadri is key to proving global generation of our line bundle. See \cite[Lemma 3.1, Remark 3.2]{seshadri} for a proof. Note that if $E$ and $F$ are locally free sheaves on a smooth, projective curve $C$ of genus $g$, then the Riemann-Roch theorem states
\begin{align*} 
    \chi(C, E \otimes F) & = \deg(E \otimes F) + \rk(E \otimes F)(1-g) \\ 
    & = \rk(E) \deg(F) + (\deg(E) + \rk(E)(1-g))\rk(F).
\end{align*}
\begin{lem}\label{seshadrimainlemma}
    Let $C$ be a smooth, projective, connected curve of genus $g \ge 2$, and let $F$ be a slope-semistable locally free sheaf on $F$. Let $r >0$ and $d$ be integers such that
    \[ r \deg F + (d + r(1-g)) \rk F = 0. \]
    If $r$ is sufficiently large, then for any line bundle $L$ of degree $d$, there exists a locally free sheaf $E$ with $\rk E = r$ and $\det E \cong L$, and
    \[ H^0(C, E \otimes F) = H^1(C, E \otimes F) = 0. \]
\end{lem}

\begin{rmk}
    The converse of Lemma \ref{seshadrimainlemma} also holds. More precisely, if $F$ is a coherent sheaf on a smooth, projective curve $C$ and there exists a locally free sheaf $E$ such that
    \[ H^0(C, E \otimes F) = H^1(C, E \otimes F) = 0, \]
    then $F$ is slope-semistable. See \cite[Theorem 2.13]{MS} for a proof.
\end{rmk}

With these preparations, we are ready to prove that for each $\C$-point $t \in \sM^\si(v)$, some power of $\la_\sE(u)$ has a global section not vanishing at $t$. 

\begin{prop}\label{globalgen}
    Let $u \in K(X)$ be as in Proposition \ref{nefclassonverticalwall}. For every $\C$-point $t_0 \in \sM^\si(v)$, there exist integers $a, m > 0$ and a global section of the line bundle $\la_\sE(u)^{m a^2}$ on $\sM^\si(v)$ that does not vanish at $t_0$.
\end{prop}
\begin{proof}
    Let $E_0 \in \sA$ be the $\si$-semistable object corresponding to $t_0$. By Lemma \ref{restsemistable}, for some $a > 0$, there exists a smooth, connected curve $C \in |a H|$ such that the derived restriction $E_0|_C^\LL$ is a slope-semistable torsion-free sheaf on $C$.
    
    Recall from Proposition \ref{nefpowerfromcurves} that associated to $C \subs X$ is the class $w \in K(X)$, and that the class 
    \[ -w|_C = \chi(v|_C \cdot h|_C) \cdot 1 - \chi(v|_C) \cdot h|_C \in K(C) \] 
    has positive rank. For any integer $m > 0$, the class $-m w|_C \in K(C)$ is determined by its rank and determinant, and so it follows from Lemma \ref{seshadrimainlemma} that for sufficiently large $m$, there exists a locally free sheaf $G$ on $C$ of class $-m w|_C$ with the property that
    \[ H^0(C, G \otimes E_0|_C^\LL) = H^1(C, G \otimes E_0|_C^\LL) = 0. \]
    
    Now consider the diagram:
    \begin{center}
    \begin{tikzpicture}
    \matrix (m) [matrix of math nodes, row sep=4em, column sep=0.5em]
    { & \sM^\si(v) \times X & & \sM^\si(v) \times C & \\
    \sM^\si(v) & & X & & C \\};
    \path[left hook->] 
    (m-1-4) edge node[auto,swap] {$ j $} (m-1-2)
    (m-2-5) edge node[auto] {$ i $} (m-2-3)
    ;
    \path[->]
    (m-1-2) edge node[auto,swap] {$ p $} (m-2-1)
    (m-1-2) edge node[pos=0.9,yshift=10pt] {$ q $} (m-2-3)
    (m-1-4) edge node[pos=0.7,yshift=-8pt] {$ p_C $} (m-2-1)
    (m-1-4) edge node[auto] {$ q_C $} (m-2-5)
    ;        
    \end{tikzpicture}
    \end{center}
    As before, let $\sE$ denote the universal complex on $\sM^\si(v) \times X$ and $\sE_C$ its restriction to $\sM^\si(v) \times C$. By Lemma \ref{nefpowerfromcurves}, we have
    \[ \la_{\sE_C}(G)^\vee = \la_{\sE_C}(-m w|_C)^\vee = \la_{\sE_C}(m w|_C) = \la_{\sE_C}(w|_C)^m = \la_\sE(u)^{m a^2}. \]
    We will apply Lemma \ref{detsection} to the complex $\sE_C$ and the sheaf $G$ to obtain a global section $\de_G$ of $\la_\sE(u)^{m a^2}$ that is nonvanishing at $t_0 \in \sM^\si(v)$. Notice that the condition of Lemma \ref{detsection}(b) holds at $t_0$ by the choice of $G$, so to conclude the proof we only have to verify the conditions of Lemma \ref{detsection}(a).
    
    Fix a $\C$-point $t \in \sM^\si(v)$, and let $E_C = \sE_C|_{\{t\} \times C}^\LL$ denote the restriction to the fiber $\{t\} \times C$. Since $G$ is locally free, we have
    \begin{equation}\label{tensorcoho}
    \sH^i(E_C \otimes G) = \sH^i(E_C) \otimes G \quad \mathrm{for\;all\;} i.
    \end{equation}
    Since $E_C$ is the restriction of $\sE|_{\{t\}\times X}^\LL$ to $C \subs X$, by Lemma \ref{restsingle}, $E_C$ fits in a triangle
    \[ \sH^0(E_C) \to E_C \to \sH^1(E_C)[-1], \]
    where $\sH^1(E_C)$ is a torsion sheaf, and by (\ref{tensorcoho}) the same is true for $E_C \otimes G$. Thus, by Lemma \ref{hypercohovanishing}, we have $\Hh^i(C, E_C \otimes G) = 0$ if $i \neq 0, 1$. Moreover, by assumption $E_C$ has class $v|_C \in K(C)$, and so we obtain
    \begin{align*} 
        \chi(v|_C \cdot [G]) & = \chi(v|_C \cdot (-m w|_C)) \\
        & = m \chi(v|_C \cdot(\chi(v|_C \cdot h|_C) \cdot 1 - \chi(v|_C) \cdot h|_C)) \\
        & = m (\chi(v|_C \cdot h|_C) \chi(v|_C)  - \chi(v|_C) \chi(v|_C \cdot h|_C)) = 0.
    \end{align*}
    Thus, the conditions of Lemma \ref{detsection}(a) hold.
\end{proof}

\subsection{Proof of projectivity}
We will now use Proposition \ref{globalgen} to prove that the line bundle $\la_\sE(u)$ on $\sM^\si(v)$ descends to a semiample line bundle on the good moduli space $M^\si(v)$ and deduce that $M^\si(v)$ is projective.

\begin{thm}\label{projectivity}
    Let $(X, H)$ be a smooth, projective, polarized surface, $v \in \Kn(X)$ a class of positive rank, and $\si = (\sA, Z)$ a stability condition on the vertical wall for $v$ as in Section \ref{section:stabvertwall}. Let $\sM^\si(v)$ be the moduli stack of $\si$-semistable objects of class $v$ in $\sA$. The good moduli space $M^\si(v)$ of $\sM^\si(v)$ is projective, and the natural nef class $L_\si$ is ample.
\end{thm}
\begin{proof}
    Let $\sE$ denote the universal complex on $\sM^\si(v) \times X$, and let $u \in K(X)$ be as in Proposition \ref{nefclassonverticalwall}. By Proposition \ref{globalgen}, for each $\C$-point $t \in \sM^\si(v)$, there exists an integer $N_t > 0$ and a global section $\de_t$ of $\la_\sE(u)^{N_t}$ that does not vanish at $t$. Since $\sM^\si(v)$ is quasicompact, there exists a single $N$ such that the line bundle $\la_\sE(u)^N$ is generated by finitely many global sections $\de_0,\ldots,\de_n \in \Ga(\sM^\si(v), \la_\sE(u)^N)$. 
    
    By Lemma \ref{nefdescendtogms} and the equivalence of categories of Proposition \ref{vbtogms}, the line bundle $\la_\sE(u)^N$ and the sections $\de_0,\ldots,\de_n$ descend to a line bundle $L$ and generating sections $\si_0,\ldots,\si_n \in \Ga(M^\si(v), L)$ on the good moduli space $M^\si(v)$ and induce a morphism $\pi: M^\si(v) \to \p^n$.
    
    We claim that $\pi$ has finite fibers. If not, there is a smooth, projective curve $C$ and a nonconstant morphism $g: C \to M^\si(v)$ such that the composition 
    \[ C \xrightarrow{g} M^\si(v) \xrightarrow{\pi} \p^n \]
    is constant. This implies that one of the sections $g^*\si_i$ is nowhere vanishing, implying that $g^*L \cong \Oh_C$. But by Proposition \ref{nefclassonverticalwall}, the line bundle $L$ is a positive multiple of the nef class $L_\si$ as an element of $\Num(M^\si(v))$ and so enjoys the positivity properties of Lemma \ref{gmsnef}. Thus, the line bundle $g^*L$ has positive degree since $g$ is nonconstant, a contradiction.
    
    Now $M^\si(v)$ is proper by Theorem \ref{gmsexists}, hence the map $\pi$ is proper. Thus, by Zariski's Main Theorem, $\pi$ is in particular quasi-finite, hence representable by schemes, see \cite[Chapter II, Theorem 6.15]{knutson} or \cite[\href{https://stacks.math.columbia.edu/tag/082J}{Tag 082J}]{stacks-project}. Thus, $M^\si(v)$ is in particular a scheme. Moreover, $\pi: M^\si(v) \to \p^n$ is finite, hence $L$ is ample, and we conclude that $M^\si(v)$ is projective.
\end{proof}


\section{Relationship to the Uhlenbeck compactification}
In this section we describe a bijective morphism $\Phi: M^{\mathrm{Uhl}}(v) \to M^\si(v)$ from the Uhlenbeck compactification of $\mu$-stable locally free sheaves to the good moduli space of $\si$-semistable objects, where $v \in \Kn(X)$ is a class of positive rank and $\si \in \Stab(X)$ lies on the vertical wall for $v$. To describe some context, let us consider the following diagram.

\begin{center}
    \begin{tikzpicture}
    \matrix (m) [matrix of math nodes, row sep=4em, column sep=4em]
    { \sM^{\mathrm{G}}(v) & \sM^{\mu}(v) & \sM^\si(v) \\
    M^{\mathrm{G}}(v) & M^{\mathrm{Uhl}}(v) & M^\si(v) \\};
    \path[right hook->] 
    (m-1-1) edge node[auto] {$ _\mathrm{open\, emb.} $} (m-1-2)
    (m-1-2) edge node[auto] {$ _\mathrm{open\, emb.} $} (m-1-3)
    ;
    \path[->]
    (m-1-1) edge node[auto] {$ _\mathrm{gms} $} (m-2-1)
    (m-1-2) edge node[auto] {$  $} (m-2-2)
    (m-1-3) edge node[auto] {$ _\mathrm{gms} $} (m-2-3)
    (m-2-1) edge node[auto] {$  $} (m-2-2)
    ;
    \path[->,dashed]
    (m-2-2) edge node[auto] {$ \Phi $} (m-2-3)
    ;
    \end{tikzpicture}
\end{center}
The top row consists of open embeddings of algebraic stacks, where from left to right the stacks are respectively that of Gieseker-semistable sheaves, $\mu$-semistable sheaves, and $\si$-semistable complexes, each of numerical class $v$ of positive rank. They all contain the stack $\sM^{\mu\mhyphen\mathrm{s}, \mathrm{lf}}(v)$ of $\mu$-stable locally free sheaves of class $v$ as an open substack, which moreover coincides with the stack of $\si$-stable objects of class $v$. We denote by $\sE$ the universal complex on $\sM^\si(v)$, and by $\sE_\mu$ its restriction to $\sM^{\mu}(v)$; this restriction is the universal sheaf. The vertical maps $\sM^{\mathrm{G}}(v) \to M^{\mathrm{G}}(v)$ and $\sM^\si(v) \to M^\si(v)$ are good moduli space maps.

The scheme $M^{\mathrm{Uhl}}(v)$ together with the middle vertical map was constructed by Li in \cite{li}, and stack-theoretically can be described as the projective spectrum 
\[ M^{\mathrm{Uhl}}(v) = \Proj \left(\bigoplus_{n \ge 0} \Ga(\sM^\mu(v), \la_{\sE_\mu}(u)^{\otimes n}) \right) \]
of the section ring of the line bundle $\la_{\sE_\mu}(u)$ on $\sM^\mu(v)$, where $u$ is as in Proposition \ref{nefclassonverticalwall}. It is not a good moduli space of $\sM^\mu(v)$, but its closed points naturally parameterize $\mu$-semistable sheaves up to the following equivalence relation. If $F$ is a torsion-free sheaf on $X$, it embeds into its double dual with cokernel $T$ supported in dimension 0:
\[ 0 \to F \to F^{\vee\vee} \to T \to 0. \]
Let $l_p(T)$ denote the length of the stalk $T_p$ as an $\Oh_{X,p}$-module at a closed point $p \in X$. Recall that if $F$ is $\mu$-semistable, we denote by $\gr(F)$ the direct sum of its Jordan-H\"older factors. Two $\mu$-semistable sheaves $F_1$ and $F_2$ correspond to the same point in $M^{\mathrm{Uhl}}(v)$ if and only if 
\begin{itemize}
    \item $\gr(F_1)^{\vee \vee}$ and $\gr(F_2)^{\vee \vee}$ are isomorphic, and
    \item $l_p(\gr(F_1)^{\vee\vee}/\gr(F_1)) = l_p(\gr(F_2)^{\vee\vee}/\gr(F_2))$ for all closed points $p \in X$.
\end{itemize}
As observed in \cite{LQ}, it follows from the classification of polystable objects in Proposition \ref{ss-object-vertical-classification} that the closed points of $M^{\mathrm{Uhl}}(v)$ and $M^\si(v)$ are in a set-theoretic bijection. In the next result we upgrade this bijection to a morphism of schemes.

\begin{thm}\label{uhlenbeck}
    There exists a morphism $\Phi: M^{\mathrm{Uhl}}(v) \to M^\si(v)$ that makes the above diagram commute and is bijective on points.
\end{thm}
\begin{proof}
    From Theorem \ref{projectivity} we see that the good moduli space $M^\si(v)$ is the projective spectrum
    \[ M^{\si}(v) = \Proj \left(\bigoplus_{n \ge 0} \Ga(\sM^{\si}(v), \la_\sE(u)^{\otimes n}) \right). \]
    The restriction maps
    \[ \Ga(\sM^{\si}(v), \la_\sE(u)^{\otimes n}) \to \Ga(\sM^\mu(v), \la_{\sE_\mu}(u)^{\otimes n}) \]
    give a homomorphism of graded rings
    \[ \bigoplus_{n \ge 0} \Ga(\sM^{\si}(v), \la_\sE(u)^{\otimes n}) \to \bigoplus_{n \ge 0} \Ga(\sM^\mu(v), \la_{\sE_\mu}(u)^{\otimes n}). \]
    Since the restrictions of sections of $\la_\sE(u)^{\otimes n}$ to $\sM^\mu(v)$ have no base points, this ring map induces the morphism $\Phi: M^{\mathrm{Uhl}}(v) \to M^\si(v)$. 
    
    To prove that $\Phi$ is surjective, it is enough to show that the composition $\sM^\mu(v) \hookrightarrow \sM^{\si}(v) \to M^\si(v)$ is surjective on $\C$-valued points. So let 
    \[ E = \left(\bigoplus_i F_i\right) \oplus \left(\bigoplus_j \Oh_{p_j}^{\oplus n_j}[-1]\right) \]
    be a $\si$-polystable object corresponding to a closed point of $M^\si(v)$, where the points $p_j \in X$ are distinct. Let $R_j$ be an Artinian quotient of the local ring $\Oh_{X, p_j}$ of length $n_j$. Note that the object
    \[ E' = \left(\bigoplus_i F_i\right) \oplus \left(\bigoplus_j R_j [-1]\right) \]
    is $\si$-semistable whose stable factors are the direct summands of $E$, and so $E'$ corresponds to the same closed point of $M^\si(v)$ as $E$. Let $F$ denote the polystable locally free sheaf $\oplus_i F_i$, choose a surjective map
    \[ F \twoheadrightarrow \bigoplus_j R_j, \]
    and let $E''$ denote the kernel of this surjection. The sheaf $E''$ is $\mu$-semistable of class $v$, and in the heart $\sA$ fits in the triangle
    \[ \bigoplus_j R_j[-1] \to E'' \to F. \]
    Thus, the stable factors of $E''$ with respect to $\si$ are again the direct summands of $E$, and so $E''$ corresponds to a $\C$-point of $\sM^\mu(v)$ that maps to the point corresponding to $E$ in $M^\si(v)$.
    
    To prove that $\Phi$ is injective, let $F$ be a $\mu$-polystable sheaf. Letting $T$ denote the quotient $F^{\vee\vee}/F$, we get a short exact sequence
    \[ 0 \to T[-1] \to F \to F^{\vee\vee} \to 0 \]
    in $\sA$. Now $F^{\vee\vee}$ is a $\mu$-polystable locally free sheaf, and in the Jordan-H\"older filtration of $T[-1]$ with respect to $\si$, the object $\Oh_p[-1]$ appears as a factor exactly $l_p(T)$ times for each $p \in X$. Thus, the $\si$-polystable object corresponding to $F$ is
    \[ F^{\vee\vee} \oplus \left(\bigoplus_{p \in X} \Oh_p^{\oplus l_p(T)}[-1] \right). \]
    From this description it is clear that two $\mu$-polystable sheaves map to the same point in $M^\si(v)$ if and only if they map to the same point in $M^{\mathrm{Uhl}}(v)$.
\end{proof}

\bibliographystyle{alpha}
\bibliography{bibliography}

\begin{thebibliography}{ABCH13}

\bibitem[AB13]{ABL13}
Daniele Arcara and Aaron Bertram.
\newblock Bridgeland-stable moduli spaces for {$K$}-trivial surfaces.
\newblock {\em J. Eur. Math. Soc. (JEMS)}, 15(1):1--38, 2013.
\newblock With an appendix by Max Lieblich.

\bibitem[ABCH13]{ABCH}
Daniele Arcara, Aaron Bertram, Izzet Coskun, and Jack Huizenga.
\newblock The minimal model program for the {H}ilbert scheme of points on
  {$\Bbb{P}^2$} and {B}ridgeland stability.
\newblock {\em Adv. Math.}, 235:580--626, 2013.

\bibitem[AHLH18]{AHLH}
Jarod Alper, Daniel Halpern-Leistner, and Jochen Heinloth.
\newblock Existence of moduli spaces for algebraic stacks, 2018.

\bibitem[Alp08]{AlperGMS}
Jarod Alper.
\newblock {\em Good moduli spaces for {A}rtin stacks}.
\newblock ProQuest LLC, Ann Arbor, MI, 2008.
\newblock Thesis (Ph.D.)--Stanford University.

\bibitem[AM17]{AM}
Daniele Arcara and Eric Miles.
\newblock Projectivity of {B}ridgeland moduli spaces on del {P}ezzo surfaces of
  {P}icard rank 2.
\newblock {\em Int. Math. Res. Not. IMRN}, (11):3426--3462, 2017.

\bibitem[AP06]{AP06}
Dan Abramovich and Alexander Polishchuk.
\newblock Sheaves of {$t$}-structures and valuative criteria for stable
  complexes.
\newblock {\em J. Reine Angew. Math.}, 590:89--130, 2006.

\bibitem[BM14]{BM}
Arend Bayer and Emanuele Macr\`\i.
\newblock Projectivity and birational geometry of {B}ridgeland moduli spaces.
\newblock {\em J. Amer. Math. Soc.}, 27(3):707--752, 2014.

\bibitem[Bri07]{bridgeland}
Tom Bridgeland.
\newblock Stability conditions on triangulated categories.
\newblock {\em Ann. of Math. (2)}, 166(2):317--345, 2007.

\bibitem[Bri08]{bridgelandK3}
Tom Bridgeland.
\newblock Stability conditions on {$K3$} surfaces.
\newblock {\em Duke Math. J.}, 141(2):241--291, 2008.

\bibitem[HL10]{HL}
Daniel Huybrechts and Manfred Lehn.
\newblock {\em The geometry of moduli spaces of sheaves}.
\newblock Cambridge Mathematical Library. Cambridge University Press,
  Cambridge, second edition, 2010.

\bibitem[Huy08]{huy}
Daniel Huybrechts.
\newblock Derived and abelian equivalence of {$K3$} surfaces.
\newblock {\em J. Algebraic Geom.}, 17(2):375--400, 2008.

\bibitem[KM76]{KM76}
Finn~Faye Knudsen and David Mumford.
\newblock The projectivity of the moduli space of stable curves. {I}.
  {P}reliminaries on ``det'' and ``{D}iv''.
\newblock {\em Math. Scand.}, 39(1):19--55, 1976.

\bibitem[Knu71]{knutson}
Donald Knutson.
\newblock {\em Algebraic spaces}.
\newblock Lecture Notes in Mathematics, Vol. 203. Springer-Verlag, Berlin-New
  York, 1971.

\bibitem[Li93]{li}
Jun Li.
\newblock Algebraic geometric interpretation of {D}onaldson's polynomial
  invariants.
\newblock {\em J. Differential Geom.}, 37(2):417--466, 1993.

\bibitem[Lie06]{lie06}
Max Lieblich.
\newblock Moduli of complexes on a proper morphism.
\newblock {\em J. Algebraic Geom.}, 15(1):175--206, 2006.

\bibitem[Lo12]{lo}
Jason Lo.
\newblock On some moduli of complexes on k3 surfaces, 2012.

\bibitem[LQ14]{LQ}
Jason Lo and Zhenbo Qin.
\newblock Mini-walls for {B}ridgeland stability conditions on the derived
  category of sheaves over surfaces.
\newblock {\em Asian J. Math.}, 18(2):321--344, 2014.

\bibitem[Mac14]{maciocia}
Antony Maciocia.
\newblock Computing the walls associated to {B}ridgeland stability conditions
  on projective surfaces.
\newblock {\em Asian J. Math.}, 18(2):263--279, 2014.

\bibitem[MM13]{MM}
Antony Maciocia and Ciaran Meachan.
\newblock Rank 1 {B}ridgeland stable moduli spaces on a principally polarized
  abelian surface.
\newblock {\em Int. Math. Res. Not. IMRN}, (9):2054--2077, 2013.

\bibitem[MS17]{MS}
Emanuele Macr\`\i and Benjamin Schmidt.
\newblock Lectures on {B}ridgeland stability.
\newblock In {\em Moduli of curves}, volume~21 of {\em Lect. Notes Unione Mat.
  Ital.}, pages 139--211. Springer, Cham, 2017.

\bibitem[Mum63]{mumford}
David Mumford.
\newblock Projective invariants of projective structures and applications.
\newblock In {\em Proc. {I}nternat. {C}ongr. {M}athematicians ({S}tockholm,
  1962)}, pages 526--530. Inst. Mittag-Leffler, Djursholm, 1963.

\bibitem[MYY18]{mmy}
Hiroki Minamide, Shintarou Yanagida, and K\={o}ta Yoshioka.
\newblock The wall-crossing behavior for {B}ridgeland's stability conditions on
  abelian and {K}3 surfaces.
\newblock {\em J. Reine Angew. Math.}, 735:1--107, 2018.

\bibitem[Nue16]{nuer}
Howard~J. Nuer.
\newblock {\em Moduli of {B}ridgeland stable objects on an {E}nriques surface}.
\newblock ProQuest LLC, Ann Arbor, MI, 2016.
\newblock Thesis (Ph.D.)--Rutgers The State University of New Jersey - New
  Brunswick.

\bibitem[Ses93]{seshadri}
C.~S. Seshadri.
\newblock Vector bundles on curves.
\newblock In {\em Linear algebraic groups and their representations ({L}os
  {A}ngeles, {CA}, 1992)}, volume 153 of {\em Contemp. Math.}, pages 163--200.
  Amer. Math. Soc., Providence, RI, 1993.

\bibitem[{Sta}20]{stacks-project}
The {Stacks project authors}.
\newblock The stacks project.
\newblock \url{https://stacks.math.columbia.edu}, 2020.

\bibitem[Tod08]{toda08}
Yukinobu Toda.
\newblock Moduli stacks and invariants of semistable objects on {$K3$}
  surfaces.
\newblock {\em Adv. Math.}, 217(6):2736--2781, 2008.

\end{thebibliography}

\end{document}